\documentclass[11pt, sort&compress]{elsarticle}

\usepackage{lineno} 
\usepackage{amsmath, amsthm, amscd, amsfonts, amssymb, graphicx}
 \usepackage{array,multirow}
 \usepackage{float}
\usepackage[a4paper]{geometry}

\usepackage{mathrsfs}
\usepackage[dvipsnames,table,xcdraw]{xcolor}
\usepackage{xspace}
\usepackage{float}
\usepackage{subfig}
\newtheorem{thm}{Theorem}[section]

\newtheorem{defn}{Definition}[section]

\newtheorem{lemma}{Lemma}[section]
\newtheorem{rem}{Remark}[section]
\usepackage{multicol}

\usepackage{setspace}

\usepackage{tikz}
\usepackage{pifont}

\usepackage{booktabs,dcolumn}

\usepackage[colorlinks=true]{hyperref}

\usepackage{color}
\usepackage{soul} 

\usepackage[colorinlistoftodos]{todonotes}

\usepackage[most]{tcolorbox}
\makeatletter
\def\ps@pprintTitle{%
 \let\@oddhead\@empty
 \let\@evenhead\@empty
 \def\@oddfoot{} %
 \let\@evenfoot\@oddfoot
 }
\makeatother

\begin{document}
\begin{frontmatter}
\title{Stability Analysis of  Fractional Order Memristor Synapse-coupled Hopfield Neural Network with Ring Structure}
\author[add1]{Leila Eftekhari}
\address[add1]{Department of Mathematics,  Tarbiat Modares University, Tehran, IR 14117-13116}
\author[add2]{Mohammad M. Amirian\corref{cor1}}
\ead{M.Amiriammatlob@dal.ca}
\address[add2]{Department of Mathematics and Statistics, Dalhousie University, Halifax, NS, CA B3H4R2}
\cortext[cor1]{}
\begin{abstract}
A memristor is a nonlinear two-terminal electrical element that incorporates memory features and nanoscale properties, enabling us to design very high-density artificial neural networks. To enhance the memory property, we should use mathematical frameworks like fractional calculus, which is capable of doing so.
 Here, we first present a fractional-order memristor synapse-coupling Hopfield neural network on two neurons and then extend the model to a neural network with a ring structure that consists of $n$ sub-network neurons, increasing the synchronization in the network. Necessary and sufficient conditions for the stability of equilibrium points are investigated, highlighting the dependency of the stability on the fractional-order value and the number of neurons.  Numerical simulations and bifurcation analysis, along with Lyapunov exponents, are given in the two-neuron case that substantiates the theoretical findings, suggesting possible routes towards chaos when the fractional order of the system increases. In the $n$-neuron case also, it is revealed that the stability depends on the structure and number of sub-networks.
\end{abstract}
\end{frontmatter}


\section{Introduction}

The brain as the central organ of the nervous system regulates most of the activities in our body. Brain dynamics is generally studied based on mathematical equations and frameworks, more specifically artificial neural networks models \cite{li2005complex, gaurav2021eeg, nobukawa2019atypical, nobukawa2020deterministic}. Artificial neural networks can mimic the complex behaviour of non-linear systems like fractional calculus. Fractional calculus is a generalised version of integer calculus and a specific  case of convolution integral \cite{matlob2019concepts}. Lately, there has been some promising progress in solving fractional differential equations using artificial neural networks \cite{zuniga2017solving, zuniga2018new, xiong2021spectral, zuniga2021fractal}.  Z{\'u}{\~n}iga-Aguilar et al have recently suggested a new method using neural networks, optimized by the Levenberg–Marquardt algorithm, to solve fractal-fractional differential equations with a non-singular and non-local kernel \cite{zuniga2021numerical}.
Aguilar et al used a fractionalized version of neural networks to study system identification, and they have found that their model has a higher precision with a fewer number of parameters compared to integer models \cite{aguilar2020fractional}. Among artificial neural networks models, Hopfield neural network (HNN) draws considerable attention due to its simple paradigm in designing the systems that provide us with a better understanding of human memory \cite{he2020complexity, njitacke2019plethora}.
HNN has a wide range of applications in artificial intelligence, such as machine learning, associative memory, pattern recognition including self-excited hyperchaotic and chaotic oscillations \cite{li2005hyperchaos}, period-adding bifurcations \cite{rech2015period}, and coexisting asymmetric multi-stable \cite{korn2003there} patterns, optimized calculation, VLSI and parallel realization of optical devices \cite{srinivasulu2012digital, joya2002hopfield}.

Memristors can be used to forecast electromagnetic induction in biological systems, neuronal simulations, and complex systems with learning and forgetting mechanisms \cite{thomas2013memristor}. 
Using memristors instead of resistors in the Hopfield neural network models, for instance, one can build a new system where the parameters vary according to their state variables, which is said to be a memristor-based neural network. This component can also be implemented by well-known models such as Hodgkin–Huxley \cite{hu2019dynamic, mazarei2018role}, Hindmarsh–Rose \cite{usha2019hindmarsh}, Morris–Lecar \cite{morris1981voltage}, Rulkov \cite{li2021memristive}, and Wilson neuronal model \cite{xu2022electromagnetic, yang2021energy} to study the specific phenomena of interest.
To enrich the non-linear behaviour of biological neurons, for example, H. Bao et al used a Hindmarsh–Rose neuronal model, thereby representing the complex dynamical features of the electrical activities better \cite{bao2019hidden}. It is worth noting that the action potential in the models could be affected by the choice activation functions, such as hyperbolic tangent function, exponential function, non-monotonous activation, piecewise-linear, etc \cite{xu2021piecewise}. Here we are using hyperbolic tangent function.

The memory features and nanoscale properties of the memristor enable us to design very high-density systems. This component can be used as synaptic weights in artificial neural networks due to their learning ability \cite{chen2014global}.
Depending on the question of interest in artificial neural network models, different types of network topology could be used to uncover a desirable output \cite{wang2003complex, kazemi2022phase}. One commonly-used network topology is ring topology. The ring network creates an ongoing pathway for signals through each node, connected to two other nodes, and it has been implemented in many fields such as circuits, mobile ad hoc networks and optical networks \cite{chen2007raa, mewanou2006link, tang2009pinning, jiang2018synchronization}. 

Unlike ordinary calculus, definitions in fractional calculus are not unique. In fractional calculus, we have various operators under different weight functions, giving us a wide range of opportunities to study complex systems incorporating memory properties on the one hand \cite{matlob2019concepts, amirian2020memory, khalighi2021three, saeedian2017memory, khalighi2021quantifying, amirian2022monodmemory} and rising computational barriers problems on the other hand \cite{khalighi2021new}. Some basic theorems and algebra lemmas such as chain rule and stability criteria holding for integer calculus is no longer valid  when it comes to fractional calculus \cite{podlubny1998fractional, matlob2019concepts}. Therefore, careful attention is required to control the dynamics of fractional network models. This paper aims to build up a solid mathematical framework for the stability criteria of n-dimensional fractional-order network models. To do so, we first introduce our fractional-order memristor synapse-coupled Hop-field neural network (\textsc{FMHNN}) on two neurons using fractional derivative in sense of Caputo (Eq. \ref{f1}). We chose this fractional operator since it only requires initial conditions given in terms of integer-order derivatives \cite{podlubny1998fractional}, consisting with the property of our model. Applying the ring network concept on  equations, we then expand our model to a more complex system, where each membrane is coupled by $n$ sub-networks. After building the required fundamental theoretical bases for stability analysis, we examine the criteria numerically for both systems and perform bifurcation analysis, investigating the impact of the initial condition on \textsc{FMHNN} model.


\section{Materials and Methods}
 
\subsection{Models}

\textbf{\textsc{FMHNN} Model}.
Using the definition (\ref{def:cap}), we extend  the concept of the integer-order memristor synapse-coupled Hopfield neural network (\textsc{MHNN}), represented in 2019 \cite{chen2019coexisting}, to  \textsc{FMHNN} in the following form

\begin{align}\label{eq:fmhnn_2_N}
\begin{cases}
{}^{\textsc{c}}_0 D^{\alpha}_t x_1=- x_1+b_1  \tanh (x_1)+b_2  \tanh (x_2)+k\varphi (x_1-x_2),\\
{}^{\textsc{c}}_0 D^{\alpha}_t x_2=- x_2+b_3  \tanh (x_1)+b_4  \tanh (x_2)-k\varphi (x_1-x_2),\\
{}^{\textsc{c}}_0 D^{\alpha}_t \varphi =x_1-x_2.
\end{cases}
\end{align}
where $0 <\alpha \leq 1$ is the fractional order, and  $b_i$ are constant parameters. Note that the \textsc{FMHNN} model reduces to \textsc{MHNN} model when $\alpha=1$.

In this model, when there is a potential difference between the neurons, an induction current will appear, and a two-way induction current $I_M$ flows through an active flux controlled memristor synapse as shown below
\begin{align}
\begin{cases}I_M=W(\varphi)V_M=k\varphi V_M,\\
\varphi=f(V_M)=V_M.\end{cases}
\end{align}
where $\varphi ~\text{and}~ V_M$ account for the state variable of the memristor, and the potential difference of the two neurons, respectively;  $W=k\varphi$ also explains a memductance function with memristor coupling weight $k$. $x_1$ and $x_2$ are the state variables describing the potentials between the inside and outside of the two neurons, $\tanh(x_{1})$ and $\tanh(x_{2})$ are the activation function of the two neurons, and $V_{M}=x_{1}-x_{2}$ is the potential difference of the two neurons. The coefficients represent the amplitude of a connection between two neurons. The \textsc{FMHNN}  schematic connection pattern of two neurons is shown in Fig (\ref{fig}).
\begin{figure}[!tbp] 
	\begin{center}
\includegraphics[width=7cm, height=3.8cm]{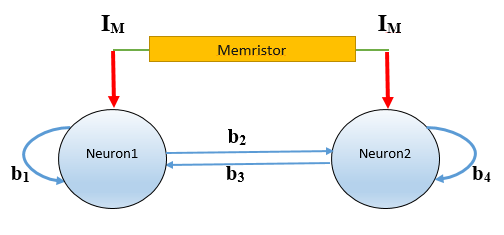}	
	\end{center}
\caption{The \textsc{FMHNN} connection pattern with two neurons.}\label{fig}
\end{figure}
\\

\textbf{\textsc{FMHNN} Model with Ring Structure}.
We apply a ring-network structure on \textsc{FMHNN} model (\ref{eq:fmhnn_2_N}), and by doing so we expand our previous model from two neurons to $n$ units of neurons. The ring network consists of $n$ sub-networks coupling strength among all membrane potentials $x_{1i}$ in the following form
\begin{align}\label{ndim}
\begin{cases}
{}^{\textsc{c}}_0 D^{\alpha}_t x_{1i} &= -x_{1i} + b_1 \tanh (x_{1i}) + b_2\tanh (x_{2i})+k\varphi_i (x_{1i}-x_{2i})+ \frac{d}{{2p}}\sum\limits_{j = i - p}^{i + p} ({{x_{1j}}}  - {x_{1i}}),
\\
{}^{\textsc{c}}_0 D^{\alpha}_t x_{2i} &= -x_{2i} + b_3 \tanh (x_{1i}) + b_4\tanh (x_{2i})-k\varphi_i (x_{1i}-x_{2i}),
\\
{}^{\textsc{c}}_0 D^{\alpha}_t\varphi _{i} &= ~~ x_{1i}-x_{2i}.
\end{cases}
\end{align}
where $ i = 1, 2, ..., n$, and $[x_{1i} , x_{2i} ,\varphi_{i}  ]$ represent the state variables of $i^{th}$ sub-network. Futher, the memristor synapses are coupled symmetrically to their $2P$ nearest neighbourhood by neural sub-networks, interacting with $P$ neighbourhood from of left and right side. 
Also, $d$ stands for the coupling strength among the $n$ sub-networks. 


\begin{defn} \cite{khalighi2021new} 
\label{def:cap}
 For a continuous function $f,$ with $ f^{'} \in L^1(R^+)$  , the Caputo fractional-order derivative of order $\alpha \in (0,1)$ is defined by 
\begin{align}
\label{f1}
{}_{t_0}^\textsc{c} D_t^\alpha f(t)=\dfrac{1}{\Gamma(1-\alpha)}\int_{t_0}^t (t-s)^{-\alpha}f'(s)ds.
 \end{align}
 \end{defn}
\subsection{Model Analysis}
Here, we study the existence and uniqueness of solutions and determine the stability criteria for both models (\ref{eq:fmhnn_2_N} \& \ref{ndim}). To do so, some fundamental definitions are required; each subsection therefore starts with those information.

\subsubsection{Existence and Uniqueness of Solutions}

\begin{defn}
The \textbf{circulant matrix} is a square matrix in which each row has the previous row components shifted cyclically one place towards the right side.
\begin{equation*}
B = circ (b_0, b_1, . . . , b_{n-1})=
\left[
{\begin{array}{*{20}{c}}
b_{0}    & b_{1} & b_{2} & b_{3} & \ddots & b_{n-2} & b_{n-1}\\
b_{n-1} & b_{0} & b_{1} & b_{2} & \ddots & b_{n-3}& b_{n-2}\\ 
b_{n-2} & b_{0} & b_{1} & b_{2} & \ddots & b_{n-4} & b_{n-3}\\
\ddots & \ddots & \ddots & \ddots &  \ddots & \ddots & \ddots \\
b_{2}     & b_{3} & b_{4} & b_{5} & \ddots & b_{0}  & b_{1} \\
b_{1}    & b_{2}  & b_{3} & b_{4} & \ddots & b_{n-1} & b_{0}
\end{array}} \right]_{n \times n}.  
\end{equation*}
Also, the \textbf{block circulant matrix} is the generalized version of the above matrix where each component consist of the circulant matrix 
\begin{equation*}
B = bcirc (B_0, B_1, . . . , B_{n-1})= 
\left[
{\begin{array}{*{20}{c}}
B_{0}    & B_{1} & B_{2} & B_{3} & \ddots & B_{n-2} & B_{n-1}\\
B_{n-1} & B_{0} & B_{1} & B_{2} & \ddots & B_{n-3}& B_{n-2}\\ 
B_{n-2} & B_{0} & B_{1} & B_{2} & \ddots & B_{n-4} & B_{n-3}\\
\ddots & \ddots & \ddots & \ddots &  \ddots &\ddots & \ddots \\
B_{2}     & B_{3} & B_{4} & B_{5} & \ddots & B_{0}  & b_{1} \\
B_{1}    & B_{2}  & B_{3} & B_{4} & \ddots & B_{n-1} & B_{0}
\end{array}} \right]_{n \times n}.
\end{equation*}
which $B_{i}$ (for $i=0, \cdots, n-1$) is a block matrices \cite{kaveh2011block, tee2007eigenvectors}. 
\end{defn}

\begin{thm}\label{thr1}
Assume that $\Omega=\{(x_1, x_2, \varphi) \in \mathop{R}^3: \max\{||\varphi||\} \leq \lambda\}$ and $\textsc{S}=\Omega \times [t_0, T]$ where $T<\infty$. For any initial conditions $(x_1(t_0), x_2(t_0), \varphi(t_0)) \in \Omega$, all the solutions $(x_1(t), x_2(t), \varphi(t)) \in \textsc{S}$ of system (\ref{eq:fmhnn_2_N}) are  unique for all $t \geq 0$.
\end{thm}

\begin{proof}
Let $\mathbb{F} (\textsc{x})=(\mathbb{F}_1(\textsc{x}), \mathbb{F}_2(\textsc{x}), \mathbb{F}_3(\textsc{x})))^T$ be a mapping function with $||.||$ norm so that 
\begin{align}
    \mathbb{F}_1(\textsc{x})&= -x_1 + b_1 \tanh (x_1) + b_2 \tanh (x_2) + k\varphi (x_1-x_2),\nonumber\\
    \mathbb{F}_2(\textsc{x})&= -x_2 + b_3 \tanh (x_1) + b_4 \tanh (x_2) - k\varphi (x_1-x_2),\nonumber\\
    \mathbb{F}_3(\textsc{x})&= ~~x_1 - x_2 \nonumber
\end{align}
where $\textsc{x}=(x_1, x_2, \varphi)$. Therefore, we can re-write our mapping function in a multi-variable form, $\mathbb{F}(\textsc{x})= A \textsc{x} + B \tanh(\textsc{x}) + \varphi K \textsc{x}$, such that 
\begin{align*}
A=
\left[ {\begin{array}{*{20}{c}}
    -1 & 0 & 0 \\ 
    0 & -1 & 0 \\ 
    1 & -1 & 0 
\end{array}} \right], ~~
B = 
\left[ {\begin{array}{*{20}{c}}
    b_1 & b_2 & 0 \\ 
    b_3 & b_4 & 0\\ 
    0   & 0   & 0
\end{array}} \right], ~~
K=
\left[ {\begin{array}{*{20}{c}}
    k & -k  & 0\\ 
   -k &  k  & 0\\ 
    0 &  0  & 0
\end{array}} \right], ~~
\textsc{x} =
\left[ {\begin{array}{*{20}{c}}
    x_1   \\ 
    x_2  \\
    \varphi
\end{array}} \right]
\end{align*}
We now prove that the system (\ref{eq:fmhnn_2_N}) satisfies the locally Lipschitz condition \cite{amirian2020memory}, i.e.
$$ \forall \textsc{x}, \bar{\textsc{x}} \in \textsc{S}, ~ \exists L \geq 0 ~~~\text{s.t}~~ ||\mathbb{F}(\textsc{x})-\mathbb{F}(\bar{\textsc{x}})|| \leq L||\textsc{x}-\bar{\textsc{x}}||$$
where $ \textsc{x}=(x_1,x_2, \varphi)^T ~\text{and}~ \bar{\textsc{x}}=(\bar{x_1},\bar{x_2}, \bar{\varphi})^T$.
\begin{align}
    ||\mathbb{F}(\textsc{x})-\mathbb{F}(\bar{\textsc{x}})||
    &=||A (\textsc{x}-\bar{\textsc{x}}) + B (\tanh(\textsc{x})-\tanh(\bar{\textsc{x}})) + K(\varphi \textsc{x}- \bar{\varphi} \bar{\textsc{x}})|| \nonumber\\
    &\leq ||A||~ ||\textsc{x}-\bar{\textsc{x}}|| + ||B||~ ||\tanh(\textsc{x})-\tanh(\bar{\textsc{x}})|| + ||K \varphi||~ ||\textsc{x}-\bar{\textsc{x}}|| \nonumber\\
    &\leq ( ||A||+ ||B|| + ||K||~|| \varphi||) ~||\textsc{x}-\bar{\textsc{x}}||  \nonumber\\
    &\leq ( ||A||+ ||B|| + \lambda||K||)~ ||\textsc{x}-\bar{\textsc{x}}||  \nonumber\\
    &=  L||\textsc{x}-\bar{\textsc{x}}|| \nonumber
    \end{align}
where $L=\left(||A||+ ||B|| + \lambda||K||\right)>0$. This means that with an initial condition $\textsc{x}(t_0)=\left(x_1(t_0), x_2(t_0), \varphi(t_0) \right)^T$, \textsc{FMHNN} model (\ref{eq:fmhnn_2_N}) has an unique solution $\textsc{x}(t) \in \textsc{S}$.
\end{proof}

\begin{thm}\label{th:uniq_ND}
Assume that $\Omega=\{(\textsc{x}_1, \textsc{x}_2, \phi) \in \mathbb{R}^{n \times 3}, ~\text{s.t}~~ \textsc{x}_1 = [x_{1i}]_{{}_{n \times 1}},~ \textsc{x}_2 = [x_{2i}]_{{}_{n \times 1}},~ \phi = [\varphi_i]_{{}_{n \times 1}}, ~ \Lambda = [\lambda_i]_{{}_{n \times 1}}, ~ \max\{||\phi||\} \leq \Lambda\}_{|_{i=1,2, ..., n}}$ and also $S=\Omega \times [\textsc{t}_0, \textsc{t}]$ where $\textsc{t}=[T_i]_{{}_{n \times 1}}$, $\textsc{t}_0=[t_{0i}]_{{}_{n \times 1}}$ for $i=1,2, ..., n$ and $\textsc{t}<\infty$. For any initial conditions $\left(\textsc{x}_1(\textsc{t}_0), \textsc{x}_2(\textsc{t}_0), \phi(\textsc{t}_0)\right) \in \Omega$, all the solutions $(\textsc{x}_1(\textsc{t}), \textsc{x}_2(\textsc{t}), \phi(\textsc{t})) \in S$ of system (\ref{ndim}) are  unique for all $\textsc{t} \geq 0$.
\end{thm}
\begin{proof}
Let $\mathbb{F} (\textsc{X})=(\mathbb{F}_1(\textsc{X}), \mathbb{F}_2(\textsc{X}), \mathbb{F}_3(\textsc{X}))^T$ be a multi-variable  mapping function with $||.||$ norm so that 
\begin{align}
    \mathbb{F}_1(X)&= 
    -\textsc{x}_1 + {b}_1 \tanh (\textsc{x}_1) + {b}_2\tanh (\textsc{x}_2)+k\varphi (\textsc{x}_1-\textsc{x}_2) +  \frac{d}{{2p}}\sum\limits_{j = i - p}^{i + p} ({{\textsc{x}_{1j}}} - {\textsc{x}_{1i}}),\nonumber \\
    \mathbb{F}_2(X)&= -\textsc{x}_2 + {b}_3 \tanh (\textsc{x}_1) + {b}_4\tanh (\textsc{x}_2)-k\varphi (\textsc{x}_1-\textsc{x}_2), \\
    \mathbb{F}_3(X)&= ~~\textsc{x}_1 - \textsc{x}_2. \nonumber
\end{align}
Therefore, we can re-write our mapping function in the following form 
$$\mathbb{F}(\textsc{X})= A X + B_1 \tanh(\textsc{X}) + B_2 \tanh(\textsc{X})+\varphi K \textsc{X},$$
where $A=bcirc(A_0,A_1,...,A_1) \in {\mathbb{R}^{n \times n}}, B_1=[\beta_1]_{1 \times n}, B_2= [\beta_2]_{1 \times n}, \beta_1 = [b_1, b_3, 0], \beta_2=[b_2, b_4, 0], ~\text{and} ~K=bcirc(\kappa,0,..,0)_{1\times n}$  so that $A_0,A_1 \in {\mathbb{R}^{3 \times 3}}$  and
\begin{equation*}
A_0=\left[ {\begin{array}{*{20}{c}}
  { - 1 - \frac{d}{p}}&0&0 \\ 
  0&{ - 1}&0 \\ 
  1&{ - 1}&0 
\end{array}} \right],~~~
A_1=\left[ {\begin{array}{*{20}{c}}
  {\frac{d}{{2p}}}&0&0 \\ 
  0&0&0 \\ 
  0&0&0 
\end{array}} \right],~~~
\kappa =\left[ {\begin{array}{*{20}{c}}
   - k &k&0 \\ 
  k&-k&0 \\ 
  0&0&0 
\end{array}} \right]
\end{equation*}
We now prove that the system (\ref{ndim}) satisfies the locally Lipschitz condition \cite{amirian2020memory}, i.e.
$$ \forall \textsc{X}, \bar{\textsc{X}} \in \textsc{S}, ~ \exists L \geq 0 ~~~\text{s.t}~~ ||\mathbb{F}(\textsc{X})-\mathbb{F}(\bar{\textsc{X}})|| \leq L||\textsc{X}-\bar{\textsc{X}}||,$$
\begin{align}
    ||\mathbb{F}(\textsc{X})-\mathbb{F}(\bar{\textsc{X}})||
    &=||A (\textsc{X}-\bar{\textsc{X}}) + (B_1 + B_2) (\tanh(\textsc{X})-\tanh(\bar{\textsc{X}})) + K(\varphi \textsc{X}- \bar{\varphi} \bar{\textsc{X}})|| \nonumber\\
    &\leq ||A||~ ||\textsc{X}-\bar{\textsc{X}}|| + ||B_1 + B_2||~ ||\tanh(\textsc{X})-\tanh(\bar{\textsc{X}})||+ ||\Lambda||~ ||K||~ ||\textsc{X}-\bar{\textsc{X}}|| \nonumber\\
    &\leq  \left(||A|| +||B_1|| +||B_2|| + ||\Lambda||~ ||K|| \right)   ||\textsc{X}-\bar{\textsc{X}}||  \nonumber\\
    &=  L||\textsc{X}-\bar{\textsc{X}}||, \nonumber
\end{align}
where $L=(||A|| + ||B_1|| + ||B_2|| + ||\Lambda||~||K||)$. This means that with an initial condition $\textsc{X}(t_0)=\left(\textsc{x}_1(\textsc{t}_0), \textsc{x}_2(\textsc{t}_0), \phi(\textsc{t}_0) \right)^T$, \textsc{FMHNN} model with ring structure (\ref{ndim}) has an unique solution $\textsc{X}(\textsc{t}) \in \textsc{S}$.
\end{proof}
\subsubsection{Stability of Solutions}
Here, we build up the theoretical framework for the stability analysis of our models.

\begin{defn}\label{def:laplace}
\cite{matlob2019concepts} For $0<\alpha<1$, the Laplace transform of Caputo fractional derivative,  ${}^{\textsc{c}}_{t_0} D^{\alpha}_t f(t)$,  derives from the following equation
$$ \int\limits_0^\infty  {{e^{ - st}}}~
 {}^{\textsc{c}}_{t_0} D^{\alpha}_t f(t) dt = {s^\alpha }F(s) - s^{\alpha - 1}f(t_0)$$ 
\end{defn}
\begin{defn}
\cite{matlob2019concepts}
For matrix $A \in {R^{n \times n}}$, the matrix Mittag-Leffler function is defined by\\
$$E_{\alpha, \beta}[A] = \sum\limits_{k = 0}^\infty  {\frac{{{A^k}}}{{\Gamma (\alpha k + \beta )}}}$$
\end{defn}

\begin{lemma}\label{lemma1}
\cite {matignon1996stability}
The linear autonomous system ${}^{\textsc{c}}_{t_0} D^{\alpha}_t \textsc{x} = A \textsc{x}$ is asymptotically stable if and only if 
\begin{equation}\label{9}
    \vert\mathrm{arg}(\lambda)\vert>\dfrac{\alpha \pi}{2}  ~when~ \lambda\in\sigma(A)
    \end{equation}
where $\sigma(A)$ denotes the spectrum of the matrix $A \in R^{n \times n}$. 
 \end{lemma}
 
 
 \begin{lemma}\label{lemma2}
 \cite{kaslik2012nonlinear}
 The roots of the second order polynomial $p(\lambda)= \lambda^2 + a \lambda + c$ satisfies inequality (\ref{9}) if and only if 
 $c > 0$ and $b/\sqrt{c} < 2\cos(\alpha \pi / 2)$.
 \end{lemma}

\begin{lemma}\label{3}
\cite{podlubny1998fractional} If $A \in {R^{n \times n}}$, $\beta$ is an arbitrary real number, and $\frac{{\alpha \pi }}{2} < \mu  < \min \{ \pi ,\alpha \pi \}$, then a real positive constant, $ {r} > 0$, exists so that
\begin{equation}
\left\| E_{\alpha, \beta}[A] \right\| \leqslant \frac{{{r}}}{{1 + \left\| A \right\|}}, \hspace{2cm} \mu \leq |\arg (\lambda_i[A])| \leq \pi, ~\text{for}~ i=1, 2, ..., n
\end{equation}
where $\lambda_i(A)$ denotes the eigenvalues of matrix A.
\end{lemma}
\begin{lemma}\label{4}
\cite{wen2008stability} Assume that $u(t)$ and $f(t)$ are real-valued piecewise-continuous functions defined on the real interval $[a, b]$, $K(t)$ is also real-valued and $K(t)\in L(a, b)$ and $u(t)$, $K(t)$ are non-negative on this interval. 
\[ \forall t \in [a, b], ~~~   u(t) \leqslant f(t) + \int\limits_a^t {K(\tau )u(\tau )d\tau } \]
and 
\[u(t) \leqslant f(t) + \int\limits_a^t {f(\tau )K(\tau )u(\tau )\exp \left\{ {\int\limits_\tau ^t {K(s)ds} } \right\}d\tau }.\]
\end{lemma}
\begin{thm}
Assume that we are given the following fractional order system with non-zero initial condition, $\textsc{x}(t_0)=\textsc{x}_0 :$ 
\begin{align}\label{eq:nD_theory}
{}_{t_0}^\textsc{c} D_t^\alpha \textsc{x}(t) = A\textsc{x}(t) + \mathbb{H}\left[\textsc{x}(t) \right]
\end{align}
where $0 < \alpha \leq 1$, $A \in {R^{n \times n}}$ is a  constant matrix, $\textsc{x} = (x_1, x_2, . . . , x_n)^T \in {R^{n \times 1}}$ is the state vector of our fractional-order system, ${}_{t_0}^\textsc{c} D_t^\alpha \textsc{x}(t)$, which consists of the linear $A\textsc{x}(t)$ and nonlinear $\mathbb{H}\left[\textsc{x}(t) \right]$ parts. System (\ref{eq:nD_theory}) is locally asymptotically stable at its equilibrium point, $\textsc{x}^\star$, when 
$$| \arg(\lambda_i)|  > \alpha \dfrac{\pi}{2}, ~~\text{and}~~\mathop {\lim }\limits_{\textsc{x} \to \textsc{x}^\star} \frac{||~ {\mathbb{H}\left[\textsc{x}(t)\right]}~ ||}{{\left\| \textsc{x}-\textsc{x}^\star \right\|}} = 0$$
where $\lambda_i$ ($i=1,2,..,n $) is the eigenvalues of matrix $A$.
\end{thm}\label{th:nD_stability}
\begin{proof}
The equation (\ref{eq:nD_theory}) can be written in the following form after taking Laplace transform from it.

$$s^\alpha \textsc{X}(s) - s^{\alpha-1}\textsc{x}_0 = A \textsc{X}(s) + L\left\{ \mathbb{H}[\textsc{x}(t)] \right\}$$
Thus 
$$\textsc{X}(s)={(s^\alpha I - A)^{ - 1}}[{s^{\alpha  - 1}}{X_0} + L \left\{\mathbb{H}[(\textsc{x}(t)] \right\}],$$
where $I$ is an $n \times n$ identity matrix. After taking Laplace inverse transform from the above equation (Table \textsc{C}1 of \cite{matlob2019concepts}), we would have
$$\textsc{x}(t) = {E_{\alpha, 1 }}(A{t^\alpha })\textsc{x}_0 + \int\limits_0^t {{{(t - \tau )}^{\alpha  - 1}}{E_{\alpha ,\alpha }}(A{{(t - \tau )}^\alpha })\mathbb{H}[\textsc{x}(\tau)] d\tau } ,$$
Let $\textsc{x}^\star$ be the solution of system \eqref{eq:nD_theory} and 
$\lim_{\textsc{x} \to \textsc{x}^\star} \dfrac{||~ {\mathbb{H}\left[\textsc{x}(t)\right]}~ ||}{{\left\| \textsc{x}-\textsc{x}^\star \right\|}} = 0$. Therefore
$$\forall \varepsilon>0 ~ \exists \delta_0 >0 ~ \text{s.t} ~ \left\| \textsc{x}(t)-\textsc{x}^\star \right\|< \delta_0 ~~\Rightarrow ~~ 
||~ {\mathbb{H}\left[\textsc{x}(t)\right]}~ || < \varepsilon \left\| \textsc{x}-\textsc{x}^\star \right\| $$
Now let $\delta $ be chosen arbitrarily subject to $0<\delta<\delta_0$ and consider solutions for which
$ \left\| \textsc{x}_0 \right\| <\delta $. From lemma \eqref{3}
$$ \exists ~ r_0, r \in R^+ ~ \text{s.t}~ 
\left\| \textsc{x}(t)- \textsc{x}^\star \right\| \leqslant \frac{{{r_0}\left\| {{\textsc{x}_0}} \right\|}}{{1 + \left\| A \right\|{t^\alpha }}} +\int\limits_0^{{t}} \frac{ \varepsilon r_0 r (t - \tau )^{\alpha  - 1} }{ \left( {1 + \left\| A \right\|(t - \tau )^\alpha} \right)} \left\| {\textsc{x}(\tau)}-\textsc{x}^* \right\| d\tau,$$
Using Grönwall's inequality Lemma (\ref{4}) and letting $\epsilon=\varepsilon r_0 r \delta_0 $, we would have

$$\left\| {\textsc{x}(t)}- \textsc{x}^* \right\| \leqslant \frac{{{r_0 }\left\| {{\textsc{x}_0}} \right\|}}{{1 + \left\| A \right\|{t^\alpha }}} +
\int\limits_0^{{t}} {\frac{{{\epsilon}\left\| {\textsc{x}(t)} \right\|}}{{1 + \left\| A \right\|{\tau ^\alpha }}}}
{\frac{{  {{{ (t - \tau )}^{\alpha  - 1}}}}}{{(1 + \left\| A \right\|{{(t - \tau )}^\alpha )}}}} 
\times 
exp\left\{{\int\limits_\tau^{{t}}{\frac{{ {{{ (t - s )}^{\alpha  - 1}}} }}{{(1 + \left\| A \right\|{{(t - s )}^\alpha )}}}}}ds
\right\}
d\tau $$
$$\leqslant \frac{{{r_0}\left\| {{\textsc{x}_0}} \right\|}}{{1 + \left\| A \right\|{t^\alpha }}}+\int\limits_0^{{t}} {\frac{{{\epsilon}\left\| {{\textsc{x}_0}} \right\|{{(t - \tau )}^{\alpha  - 1}}}}{{(1 + \left\| A \right\|{\tau ^\alpha }){{(1 + \left\| A \right\|{{(t - \tau )}^\alpha })}^{1-1/\alpha \left\| A \right\|}}}}} d\tau,$$
The integral equals to the sum of the two parts
$$\int\limits_0^{t} {\frac{{{\epsilon}\left\| {{\textsc{x}_0}} \right\|{{(t-\tau )}^{\alpha  - 1}}}}{{(1 + \left\| A \right\|{\tau ^\alpha }){{(1 + \left\| A \right\|{{(t-\tau )}^\alpha })}^{1 - 1/\alpha \left\| A \right\|}}}}}d\tau = \int\limits_0^{t/2} {\frac{{{\epsilon}\left\| {{\textsc{x}_0}} \right\|{{( t-\tau )}^{\alpha  - 1}}}}{{(1 + \left\| A \right\|{\tau ^\alpha }){{(1 + \left\| A \right\|{{(t-\tau )}^\alpha })}^{1 - 1/\alpha \left\| A \right\|}}}}}d\tau $$
$$+
\int\limits_{t/2}^{t} {\frac{{{\epsilon}\left\| {{\textsc{x}_0}} \right\|{{(t - \tau )}^{\alpha  - 1}}}}{{(1 + \left\| A \right\|{\tau ^\alpha }){{(1 + \left\| A \right\|{{(t-\tau )}^\alpha })}^{1 - 1/\alpha \left\| A \right\|}}}}}d\tau  \hfill \\$$
Since $\alpha<1$ and $(t - \tau ) \geqslant \tau $ when 
 $ \tau  \in [0,t/2] $, we would have 
 $$
  \int\limits_0^{t/2} {\frac{{{\epsilon}\left\| {{\textsc{x}_0}} \right\|{{(t - \tau )}^{\alpha  - 1}}}}{{(1 + \left\| A \right\|{(t-\tau) ^\alpha }){{(1 + \left\| A \right\|{{\tau }^\alpha })}^{1 - 1/\alpha \left\| A \right\|}}}}}d\tau  \hfill \\
   \leqslant \int\limits_0^{t/2} {} \frac{{{\epsilon}\left\| {{\textsc{x}_0}} \right\|{{(\tau )}^{\alpha  - 1}}}}{{(1 + \left\| A \right\|{\tau ^\alpha }){{(1 + \left\| A \right\|{{\tau }^\alpha })}^{1 - 1/\alpha \left\| A \right\|}}}}d\tau \\
   $$
Similarly
$$\int\limits_{t/2}^t {\frac{{{\epsilon}\left\| {{\textsc{x}_0}} \right\|{{(t - \tau )}^{\alpha  - 1}}}}{{(1 + \left\| A \right\|{{\tau }^\alpha }){{(1 + \left\| A \right\|{{(t - \tau )}^\alpha })}^{1 - 1/\alpha \left\| A \right\|}}}}}d\tau \leqslant
\int\limits_{t/2}^t {\frac{{{\epsilon}\left\| {{\textsc{x}_0}} \right\|{{(t - \tau )}^{\alpha  - 1}}}}{{(1 + \left\| A \right\|{{(t - \tau )}^\alpha }){{(1 + \left\| A \right\|{{(t - \tau )}^\alpha })}^{1 - 1/\alpha \left\| A \right\|}}}}} d\tau $$
$$= \int\limits_{t/2}^t {\frac{{{\epsilon}\left\| {{\textsc{x}_0}} \right\|{{s }^{\alpha  - 1}}}}{{(1 + \left\| A \right\|{{s }^\alpha }){{(1 + \left\| A \right\|{{s}^\alpha })}^{1 - 1/\alpha \left\| A \right\|}}}}} ds $$
By substituting $ s $ for $(t - \tau )$ and  by considering  $| \arg(\lambda_i)|  > \alpha \dfrac{\pi}{2} $, we would have $ \alpha \left\| A \right\|>1$, resulting in the following inequality   
$$\left\| {\textsc{x}(t)}- \textsc{x}^* \right\| \leqslant \frac{{{r_0}\left\| {{\textsc{x}_0}} \right\|}}{{1 + \left\| A \right\|{t^\alpha }}}+ 2\int\limits_0^{t/2} \frac{{ {\epsilon}\left\| {{\textsc{x}_0}} \right\|}}{{{{(1+\left\| A \right\| \tau ^{\alpha})^{(2-1/\alpha \left\| A \right\|)} }}}} d\tau, $$
$$ = \frac{{{r_0}\left\| {{\textsc{x}_0}} \right\|}}{{1 + \left\| A \right\|{t^\alpha }}} +\frac{{2\epsilon\left\| {{\textsc{x}_0}} \right\|}}{{\alpha \left\| A \right\| - 1}} +\frac{{2\epsilon\left\| {{\textsc{x}_0}} \right\|}}{{(1 - \alpha \left\| A \right\|){{(1 + \left\| A \right\|(\frac{{{t }}}{2})}^\alpha)^{1 - 1/\alpha \left\| A \right\|}}}}
$$  
Since $\delta $ and $\delta_0 $ are arbitrarily small, we would have 
$$\mathop {\lim }\limits_{t \to \infty } 
\left(
\frac{{{r_0}\left\| {{\textsc{x}_0}} \right\|}}{{1 + \left\| A \right\|{t^\alpha }}} +\frac{{2\epsilon\left\| {{\textsc{x}_0}} \right\|}}{{\alpha \left\| A \right\| - 1}} +\frac{{2\epsilon\left\| {{\textsc{x}_0}} \right\|}}{{(1 - \alpha \left\| A \right\|){{(1 + \left\| A \right\|(\frac{{{t }}}{2})}^\alpha)^{1 - 1/\alpha \left\| A \right\|}}}} 
\right)
= 0 $$
Therefore, the solution of system (\ref{eq:nD_theory}) is asymptotically stable.
\end{proof}
\begin{rem}
Let $B=bcirc(A,C,...,C)\in {\mathbb{C}^{N \times N}} $then 
$$\lambda({B})= \left\{  \lambda (A + (N - 1)C),\underbrace {\lambda (A - C),...,\lambda (A - C)}_{N - 1} \right\}$$
\end{rem}

\subsubsection{Stability of FMHNN  Model}
 
 \begin{thm} \label{th:2D_stability}
The system  \eqref{eq:fmhnn_2_N} is asymptotically stable at its equilibrium point,  $  E^\star=[0, 0, \delta]$, if and only if one of the following conditions hold

\begin{itemize}
\item[i]. $ \eta >0 $ and $ \dfrac{{{\tau _{}}}}{{\sqrt \eta  }} < 2\cos (\dfrac{{\alpha \pi }}{2})$  
\item[ii]. $ \eta =0 $ and $\delta=-3/k$ 
\end{itemize}
where $\eta = (b_1 + k \delta -1)(b_4 - k \delta -1) - (b_3-k\delta)(b_2-k\delta)$ and $\tau = b_1 + b_4 + 2(k\delta-1)$.
\end{thm}

\begin{proof} The Jacobian matrix of the system \eqref{eq:fmhnn_2_N} is as follow
\begin{align}\label{eq2}
J\left[x_1, x_2, \varphi \right]=\begin{bmatrix}
-1 + b_1 \mathrm{sech}^2(x_1) + k\varphi  &  b_2\mathrm{sech}^2(x_2) -k\varphi & k(x_1-x_2)\\
b_3 \mathrm{sech}^2(x_1) - k\varphi & -1 + b_4\mathrm{sech}^2(x_2) + k\varphi & -k(x_1-x_2)\\
1& -1&0
\end{bmatrix}
\end{align}
which at its equilibrium, $E^\star=[0, 0, \delta]$, will be 
\begin{align}\label{genera}
J(E^\star) = 
\begin{bmatrix}
-1 + b_1 + k\delta & b_2 - k\delta &  0\\
  b_3 - k\delta & -1 + b_4 + k\delta &  0\\
1& -1&0
\end{bmatrix},
\end{align}
Thus, characteristic equation of matrix \eqref{genera} will be in the form of $p(\lambda)= \lambda p_{1}(\lambda)$ where
\begin{equation}\label{eq:p1_lambda}
p_1(\lambda)= \lambda^2 + [-b_1 - b_4 + 2(1-k\delta)] \lambda + (1-b_1 )(1-b_4 ) - b_2b_3 + (b_1+b_2+b_3+b_4-2) k \delta  
\end{equation}
Assuming  $\tau = -b_1 - b_4 + 2(1-k\delta)$ and $\eta = (1-b_1 )(1-b_4 ) - b_2b_3 + (b_1+b_2+b_3+b_4-2) k \delta $, we can re-write the characteristic equation in the form of
$$p(\lambda)= \lambda (\lambda^2 - \tau \lambda + \eta)$$
According to lemma \eqref{lemma2} then, the proof of (i) would be straightforward. Regarding case (ii), for $\eta = 0$ and $\delta = -3/k$, the characteristic equation would have one zero root, $\lambda_{1} = 0$, and one negative real root, $\lambda_{1} = -4.1$. Based on lemma  (\ref{lemma1}) therefore, $E^\star$ would be a stable point (Table \ref{t1}).
\end{proof}
\begin{table}[!h]
\begin{center}
\caption{Stability region of two-neuron \textsc{FMHNN} model (\ref{eq:fmhnn_2_N}) under $b_1=-0.1,~ b_2=2.8,~ b_3=-3,~\text{and}~ b_4=4$ assumption.}\label{t1}
\begin{tabular}{p{1.5cm}llcc}
\hline
\textbf{~~~\#} & \multicolumn{3}{c}{\textbf{Conditions on parameters ($\eta, \tau, \delta$)}} &  \textbf{Stability conditions} \\
\rowcolor{gray!20}
Case 1 & $\eta=0$  & --- & \hspace{1.2cm} $\delta=-3/k$  & Stable\\
Case 2 & $\eta>0$  & $\tau=0$ & \hspace{1.7cm} $\delta=-0.95/k$ & Stable \\
\rowcolor{gray!20}
Case 3 & $\eta>0$  & $\tau>0$ &  $-0.95/k<\delta<1.9512/k$ & Stable \\
Case 4 & $\eta>\tau^2 / 4$ & $\tau<0$ &  $ -0.95/k< \delta < 1. 9512/k$ & $\cos(\frac{\alpha\pi}{2}) >  \frac{\tau}{2\sqrt{\eta}}$\\
\rowcolor{gray!20}
Case 5 & $\eta <0$  & --- & $\delta > 1.9512/k$  ~\textsc{or} $~\delta < -3/k$ & Unstable \\
\hline
\end{tabular}
\end{center}
\end{table}
\begin{rem}
When $\eta <0$ then $\tau^2-4\eta>0$ and $\delta > 1.9512/k  ~\;\text{or}\; ~\delta < -3/k$. Therefore,  $\lambda_{1,2}$  has at least one positive real root, which shows $E^\star$ is an unstable point. 
\end{rem}

\subsubsection*{\textbf{Specified Scenario: Example I}}
Same as Chen et al \cite{chen2019coexisting}, we set $b_1=-0.1,~ b_2=2.8,~ b_3=-3 ~\text{and} ~b_4=4$ in system (\ref{eq:fmhnn_2_N}). At the equilibrium point thus, the characteristic equation \eqref{eq:p1_lambda} reduces to
$$P(\lambda) =\lambda\left[\lambda^2-(1.9+2k\delta)\lambda+5.1+1.7k\delta\right],$$

Through assumption of $\tau = -(1.9+2k\delta), ~\eta = 5.1+1.7k\delta$, and $k = 0.15$ in Theorem (\ref{th:2D_stability}), Table (\ref{t1}) simplifies into Table (\ref{t2}). Note that the stability criteria of model (\ref{eq:fmhnn_2_N}) is calculated in upper and lower bound of $\delta$ value.

 \begin{table}[!h]
\begin{center}
\caption{The amount of eigenvalues for the stability region according table (\ref{t1}) when $k=0.15$ in FMHNN model (\ref{eq:fmhnn_2_N}).}
\label{t2}
\begin{tabular}{l|ccccc}
\hline
\multirow{2}{0.7cm}{\textbf{~~\#}} &
\multirow{2}{0.5cm}{~~~\textbf{$\delta$}} & 
\multirow{2}{2.4cm}{\textbf{Equilibrium}} & 
\multirow{2}{3cm}{\textbf{Eigenvalues $\lambda_{1,2}$}} & 
\multicolumn{2}{c}{\textbf{Could be stable in}}
\\
& & & & integer case? & fractional case? \\
\hline
\rowcolor{gray!20}
Case 1 & $-20$&$(0, 0, -20)$ & \hspace{0.2cm} $0.00, -4.10$ & No & No \\
Case 2 & $-6.3$&$(0, 0, -6.3)$ &\hspace{0.2cm} $ 0.00 \pm 1.87i$ & No & \color{RubineRed}{Yes}\\
\rowcolor{gray!20}
Case 3 & $-18$&$(0, 0, -18)$ &$ -3.35, -0.15$ & Yes & Yes \\
Case 3 & $-8.0$&$(0, 0, -8.0)$ &$-0.25 \pm 1.73i$ & Yes & Yes \\
\rowcolor{gray!20}
Case 4 & $-5.0$&$(0, 0, -5.0)$ &\hspace{0.3cm}$0.20 \pm 1.95i$ & No & \color{RubineRed}{Yes}\\
Case 4 & ~~$10$&$(0, 0, ~~10)$ &\hspace{0.3cm}$2.45 \pm 1.28i$ & No & \color{RubineRed}{Yes}\\
\rowcolor{gray!20}
Case 5 & ~~$14$&$(0, 0, ~~14)$ &\hspace{0.3cm}$3.85, ~~ 2.25$ & No & No\\
Case 5 & $-25$&$(0, 0, -25)$ &\hspace{0.3cm}$0.21, -5.82$ & No & No\\
\hline
\end{tabular}
\end{center}
\end{table} 

 \subsubsection{Stability of FMHNN  Model with Ring Structure}

\begin{thm}
\label{th:stability_nSub}
The system \eqref{ndim} is locally asymptotically stable at its equilibrium point, $\textsc{X}^\star=(0, 0, \phi^\star)$, when $1<n < \left(2p(d+1)+d \right)/d$ and $0 < \alpha <1$ (Table \ref{t3j}).
 \end{thm}
\begin{proof} First we re-write the system \eqref{ndim} in the compact form as below

\begin{equation} \label{eq:comp}
    \mathbb{F}(X)= A X + \mathbb{H}(X)
\end{equation}
where $\mathbb{H}(X)= B_1 \tanh(\textsc{x}_1) + B_2 \tanh(\textsc{x}_2) + \varphi K X$,  $A=bcirc(A_0,A_1,...,A_1) \in {\mathbb{R}^{n \times n}}, B_1=[\beta_1]_{1 \times n}, B_2= [\beta_2]_{1 \times n}, \beta_1 = [b_1, b_3, 0], \beta_2=[b_2, b_4, 0], ~\text{and} ~K=bcirc(\kappa,0,..,0)_{1\times n}$  so that $A_0,A_1 \in {\mathbb{R}^{3 \times 3}}$  and
\begin{equation*}
A_0=\left[ {\begin{array}{*{20}{c}}
  { - 1 - \frac{d}{p}}&0&0 \\ 
  0&{ - 1}&0 \\ 
  1&{ - 1}&0 
\end{array}} \right],~~~
A_1=\left[ {\begin{array}{*{20}{c}}
  {\frac{d}{{2p}}}&0&0 \\ 
  0&0&0 \\ 
  0&0&0 
\end{array}} \right],~~~
\kappa =\left[ {\begin{array}{*{20}{c}}
   - k &k&0 \\ 
  k&-k&0 \\ 
  0&0&0 
\end{array}} \right]
\end{equation*}
Next we show that equation (\ref{eq:comp}) satisfies stability condition  at its equilibrium point, $\textsc{X}^\star=(0, 0, \phi^\star)$ because
\begin{equation*}
    \mathop {\lim }\limits_{X \to X^*} \frac{||\mathbb{H}(X)||}{||X -X^\star||}
    =
    \mathop {\lim }\limits_{X \to X^*} \frac{(|| B_1||~|| \tanh ({\textsc{x}_1}) || + ||B_2||~ || \tanh(\textsc{x}_2) || + || \varphi||~ || K ||~ || X ||)}{|| X -X^\star||}
    = 0
\end{equation*}
which this satisfies one of conditions of theorem (\ref{th:nD_stability}). The other condition also holds when $1<n < \left(2p(d+1)+d \right)/d$ and $0 < \alpha <1$ because
\begin{align*}
    \det(A-\lambda I)&= \left\{ {\lambda ({A_0} + (n - 1) \times {A_1}), ~\lambda ({A_0} - {A_1}) \times (n - 1)} \right\}\nonumber
    \\ 
    &=\left\{ {-1, ~0, ~\dfrac{d(n - 1)-2p(d + 1)}{2p},~ 1 - n,~ 0,~ (1 - n)({3d}/{2p}+1)} \right\}
\end{align*}
\end{proof}
\begin{table}[!h]
\begin{center}
\caption{Stability region of \textsc{FMHNN} model (\ref{ndim}). No condition is applied to the memristor variable, $\phi^\star$, at equilibrium.}\label{t3j}
\begin{tabular}{cccc}
\multicolumn{4}{c}{\textbf{Conditions on }} \\
\hline
\rowcolor{gray!20}
NO. of sub-networks &&& Fractional order  \\
$1<n < \left(2p(d+1)+d \right) / d$ & && $0<\alpha<1$ \\
\hline
\end{tabular}
\end{center}
\end{table}

\begin{rem}
\label{remark}
For $\alpha =1$, under the assumption of $1<n < \left(2p(d+1)+d \right)/d$ at equilibrium point, $\textsc{X}^\star=(0, 0, \phi^\star)$, the system (\ref{ndim}) would have limit cycles.
\end{rem}

\subsubsection*{\textbf{Specified Scenario: Example II}}
In addition to our parameter assumption in the  \textsc{FMHNN} model, we set $p =1$, and $d=0.5$. According to the compact form of model (\ref{eq:comp}) then, we would have $\beta_1 = [-0.1, -3, 0],~ \beta_2=[2.8, 4, 0]$ and
\begin{equation}
A_0=\left[ {\begin{array}{*{20}{c}}
  { - 1.1}&0&0 \\ 
  0&{ - 1}&0 \\ 
  1&{ - 1}&0 
\end{array}} \right],~
A_1=\left[ {\begin{array}{*{20}{c}}
 {0.05}&0&0 \\ 
  0&0&0 \\ 
  0&0&0 
\end{array}} \right],
\kappa =\left[ {\begin{array}{*{20}{c}}
   - 0.15 &0.15&0 \\ 
  0.15&-0.15&0 \\ 
  0&0&0 
\end{array}} \right]
\end{equation}
This parameter assumption satisfies the stability criteria represented by theorem (\ref{th:stability_nSub}) when 
\begin{equation}\label{eq:example_II}
    1 < n < (2(1)(1+0.5)+0.5) / 0.5 = 7
\end{equation}

\subsubsection{Numerical Analysis}
We used \texttt{FDE12} package \cite{garrappa2018numerical} to find the solution for the models. \texttt{FDE12} is a package written in Matlab that numerically solves the fractional ordinary differential equations. All the figures and numerical simulations are done by \href{https://www.mathworks.com/products/matlab.html}{MATLAB} programming language.

\section{Results} 

\textbf{FMHNN  Model}. We specified our theoretical calculations (Table. \ref{t1}) through the same parameterization as Chen et al \cite{chen2019coexisting}, calculated the eigenvalues for the model (\ref{eq:fmhnn_2_N}), and compared the integer, $\alpha=1$, and non-integer, $0<\alpha<1$, case with each other (Table. \ref{t2}). In the integer case, the null solution is not asymptotically stable for all cases but case 3 because of non-negative eigenvalues, while this is not the case when $0<\alpha<1$. For case 2, $\delta = -6.3$, the system  (\ref{eq:fmhnn_2_N}) is asymptotically stable when $\alpha \in (0,1)$ as lemma (\ref{lemma1}) holds for this range. With regard to case 4 that $-6.3 < \delta < 13$, we examined two equilibrium points ($\delta=-5$ and $\delta=10$). According to the amount of eigenvalues ($\lambda_{1,2}$), given in Table (\ref{t2}), at the equilibrium  points $E^\star_1=(0, 0, -5)$ and $E^\star_2=(0, 0, 10)$, the stability condition (\ref{lemma1}) satisfies when $\alpha<0.9348$ and $\alpha<0.307$ respectively. 
As for case 5, we similarly looked into two equilibrium points ($\delta=14$ and $\delta=-25$). The system is not stable in any of these two cases because the stability criteria are violated due to having positive real roots for the characteristic equation. The asymptotical stability comparison of integer and non-integer cases is summarised in Table (\ref{t2}). 

\begin{figure}[!h]  
\centering
\includegraphics[width=11cm, height=4.5cm]{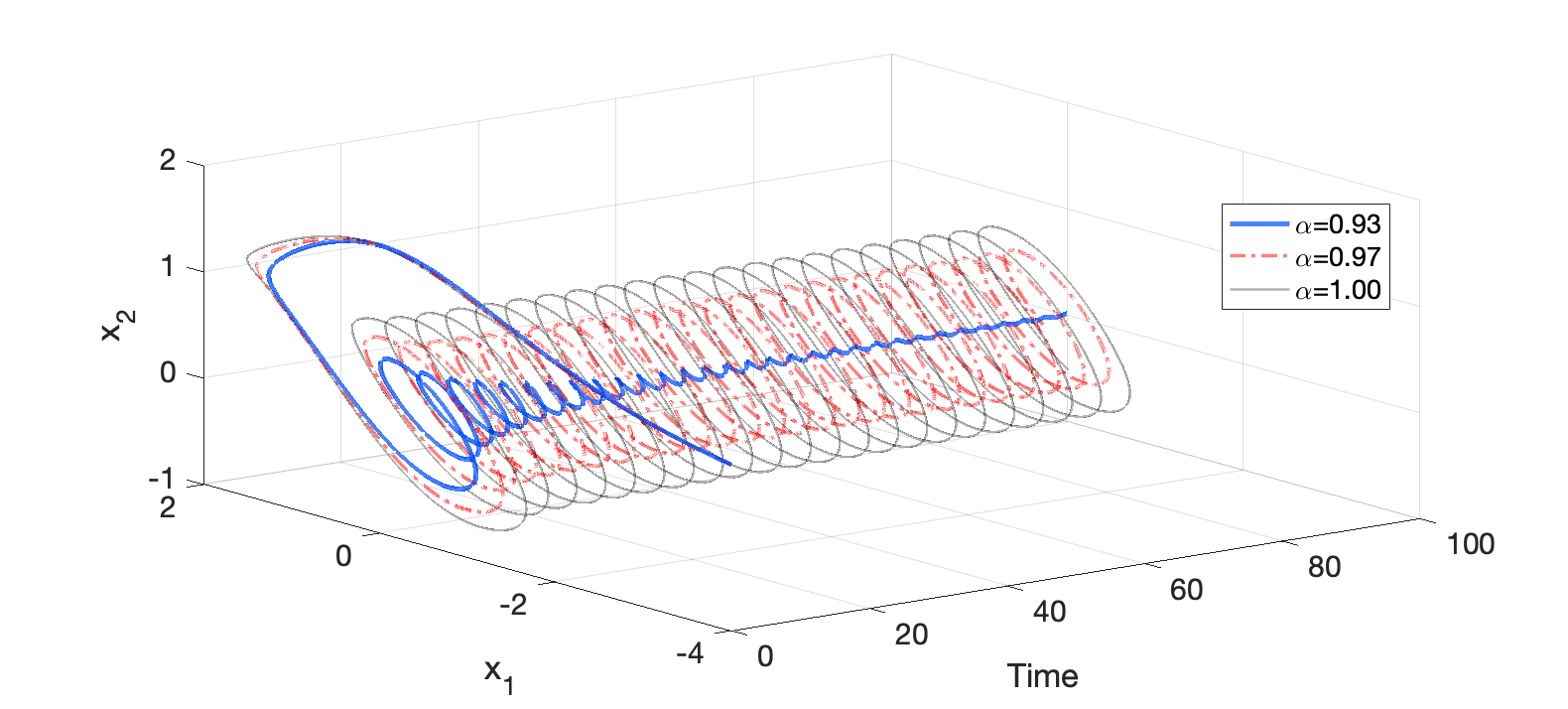}
\includegraphics[width=11.5cm,height=4.5cm]{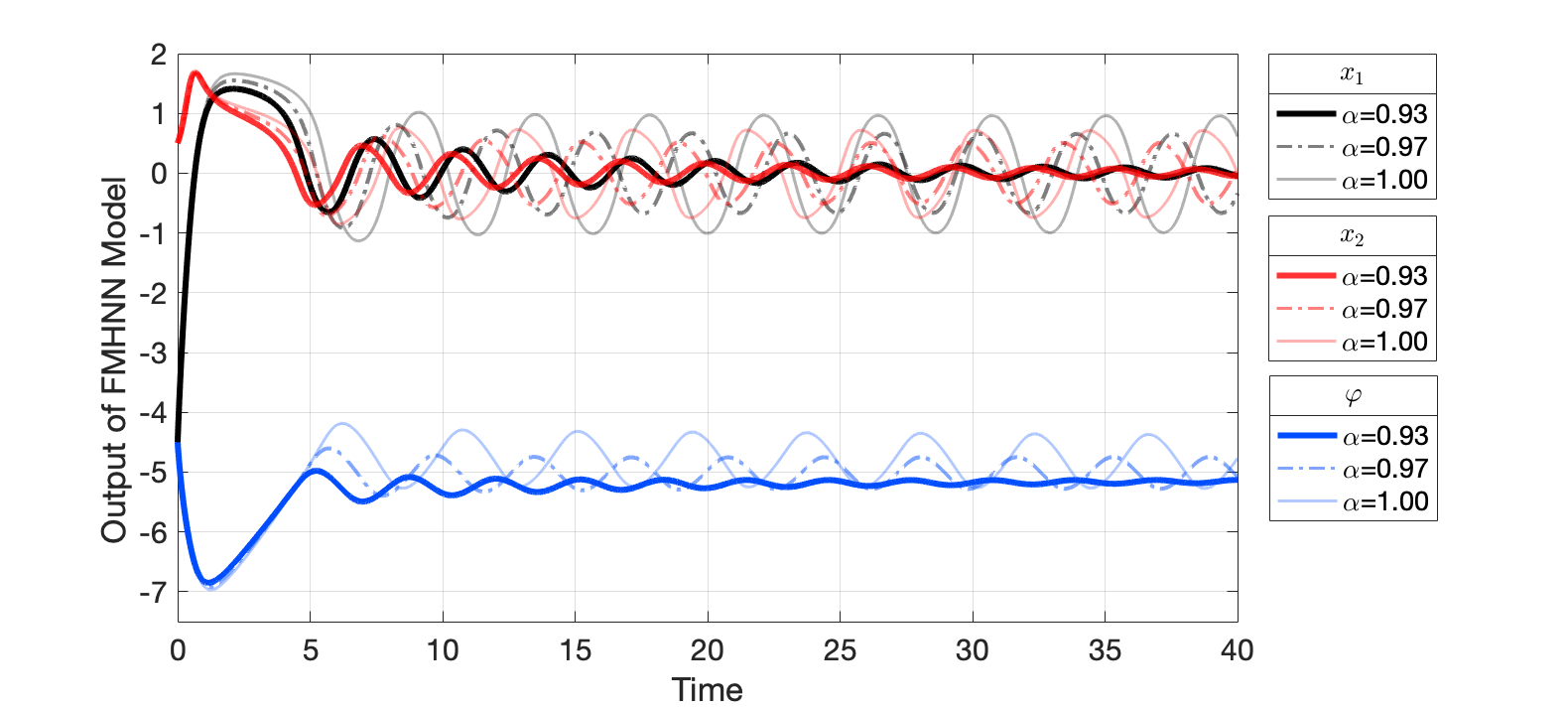}
\caption{ Phase diagrams and time series of \textsc{FMHNN} model (\ref{eq:fmhnn_2_N}) in Case 4, $\delta=-5$, with initial value $(x_1(0), x_2(0), \varphi(0)) = (-4.5,~ 0.5, -4.5)$ for three different fractional orders. Details are given in the legends. } \label{fig:case4}
\end{figure}
We next tested the calculations numerically. From the phase plane simulation for case 4, $\delta = -5$, we can see that the dynamical behavior of the model is unstable for the integer case, $\alpha=1$, and even for the non-integer case when $\alpha=0.97$ (Fig. \ref{fig:case4}, top), as the fractional order is greater than $2|arg(\lambda_{1,2})|/ \pi$, violating the stability criteria (\ref{lemma1}). However, when $\alpha=0.93$ which is somewhat less than 0.9348, the calculated value for the stability condition, the model output reaches the equilibrium point $E^\star=(0,0, -5)$ over time (Fig. \ref{fig:case4}, bottom). Regardless of the stability condition, the $\alpha$ parameter also dampens out the overall dynamics of the model. Looking at the time series simulation (Fig. \ref{fig:case4}, bottom), we can see the shading dashed lines, representing $\alpha=0.97$, have a smaller amplitude compared to the solid-shading lines, describing $\alpha=1.00$.

\begin{figure}[!h]  
\centering
\includegraphics[width=11cm, height=4.5cm]{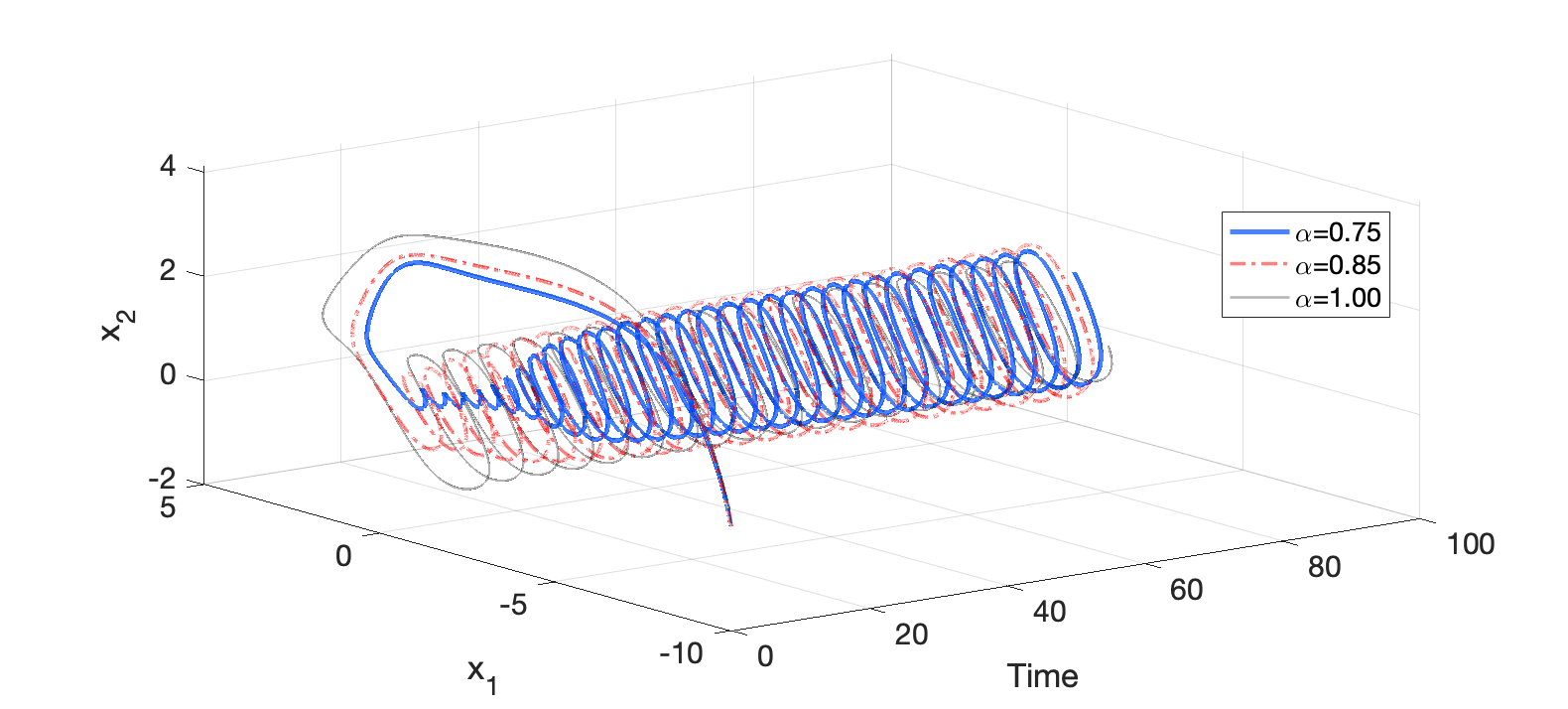}
\includegraphics[width=11.5cm,height=4.5cm]{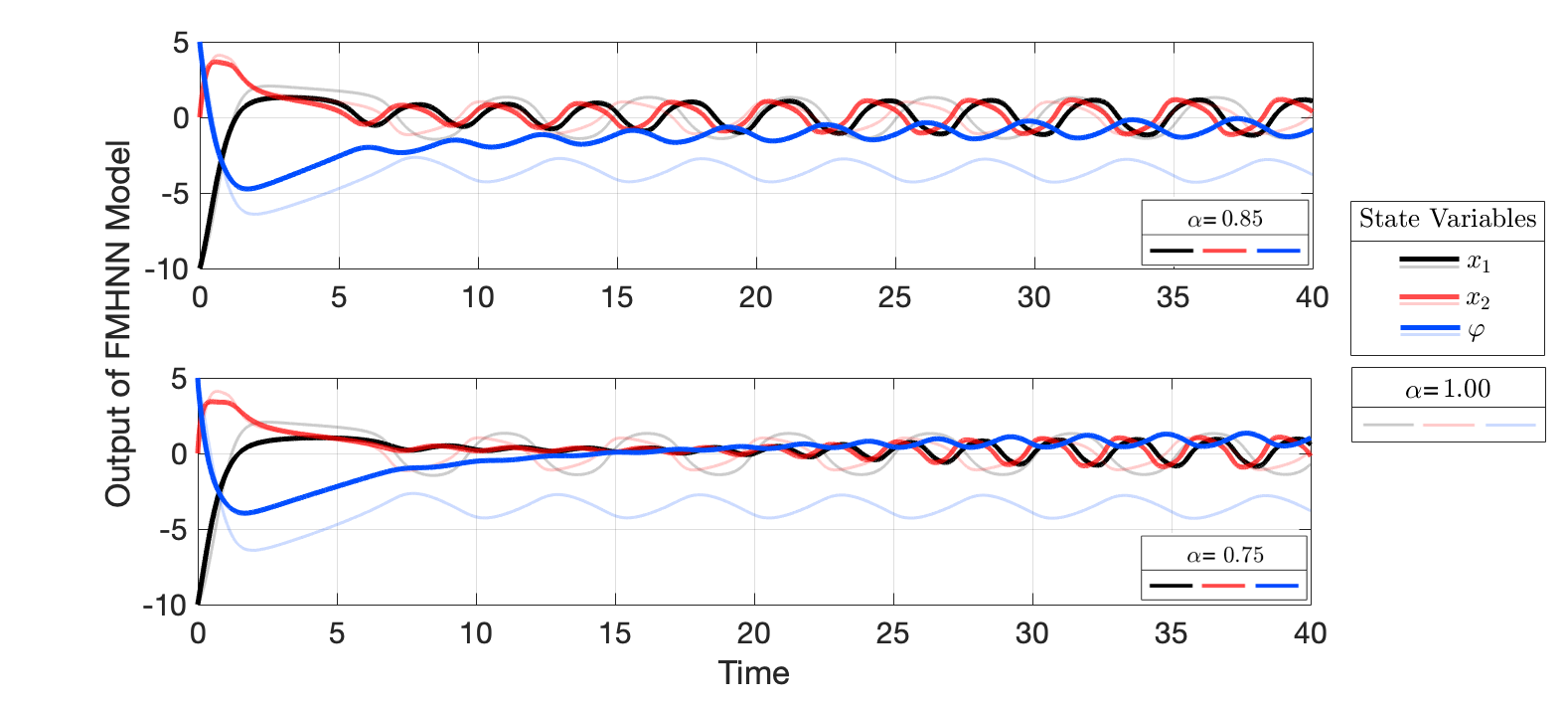}
\caption{ Phase diagrams and time series of \textsc{fmhnn} model (\ref{eq:fmhnn_2_N}) in case 5, $\delta=14$, with initial value $(x_1(0), x_2(0), \varphi(0)) = (-10,~ 10^{-6}, 5)$ for three different fractional orders. Details are given in the legends. } \label{fig:case5}
\end{figure}
As for the unstable region, case 5, $\delta=14$, we observed that the model is unstable in both integer and non-integer cases (Fig. \ref{fig:case5}). It is, however, interesting that the fractional order could cause a delay in the dynamical transition of the model output. For $\alpha=0.75$, the model starts with smooth destabilizing fluctuations after 20 seconds (Time $ > 20$), while under the same condition for the model with $\alpha=0.85$ the oscillations start about 4 times sooner, Time $\approx$ 5, (Fig. \ref{fig:case5}, bottom).


\begin{figure}[!h]  
\centering
\includegraphics[width=12cm, height=5cm]{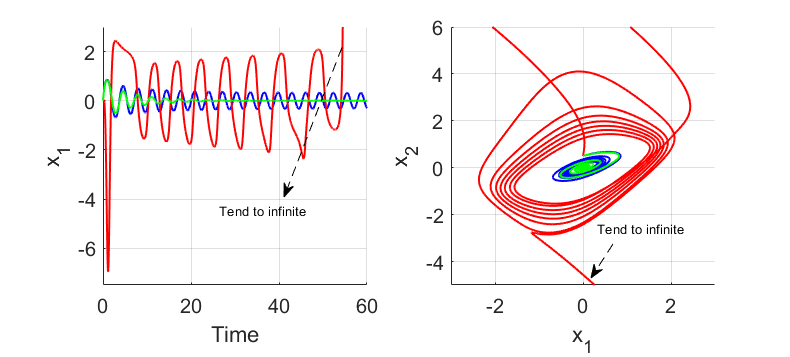}
\caption{Coexisting multi-stable patterns of the attraction regions in the \textsc{FMHNN} model (\ref{eq:fmhnn_2_N}), with initial values $ X_{\text{blue}}=(0,0.5,-5),~ X_{\text{red}}=(0,0.5,14)$, and $ X_{\text{green}}=(0,0.5,-6)$ for fractional order $ \alpha = 0.93 $.} \label{fig:bistaility}
\end{figure}

In our numerical analysis also, we noticed several striking coexisting multi-stable patterns and coexisting multiple attractors in the two-neuron-based FMHNN model (\ref{eq:fmhnn_2_N}). To elaborate upon that we performed three independent simulations using different initial values, showing the coexistence of several types of disconnected attractors due to periodic attractors with different periodicities, stable points, and unbounded divergent orbits which go to infinity. The phase plane and time series of coexisting multi-stable patterns of the attraction regions are shown in Fig. (\ref{fig:bistaility}). 

\textbf{FMHNN  Model with Ring Structure}.
In addition to considering the same model parameterization as Chen et al \cite{chen2019coexisting}, we set $(p, d) = (1, 0.5)$ and demonstrated that the stability of the model with ring structure (\ref{ndim}) depends upon the number of sub-network, n. For $n \geq 7$ we showed that the model is unstable, and we tested this through solving the model numerically for n = 7, Fig. (\ref{fig:N seven}). From this figure, we notice that the $\alpha$ parameter affects the amplitude of the oscillations. The smaller the alpha value, the smaller amplitude. 

We also performed another simulation for $n =5$ where the model parameterization meets the asymptotic stability conditions (\ref{th:stability_nSub}). In this case, the model is stable when $0<\alpha<1$, (Fig. \ref{fig:integer-nD}, top), and it would have stable limit cycles when $\alpha = 1$ (Fig. \ref{fig:integer-nD}, bottom). We also examined the impact of $\alpha$ parameter on the dynamical behavior of the model for $n=5$. Although the model would be stable for all $ \alpha$ values in (0, 1) interval, under different $\alpha$ values the trajectory of model output leads to various stationary points, Fig. (\ref{fig:NDIME}). This indicates the importance of choosing the 'right' $\alpha$ parameter in model implementation on real problems. Turing and parameter optimizations could give us an insight what the $\alpha$ value must be \cite{khalighi2021quantifying}.

\begin{figure}[!h]  
\centering
 \includegraphics[width=11cm,height=4.5cm]{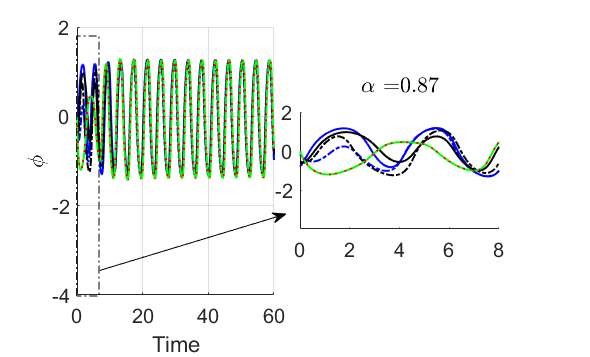}
\includegraphics[width=10cm,height=5cm]{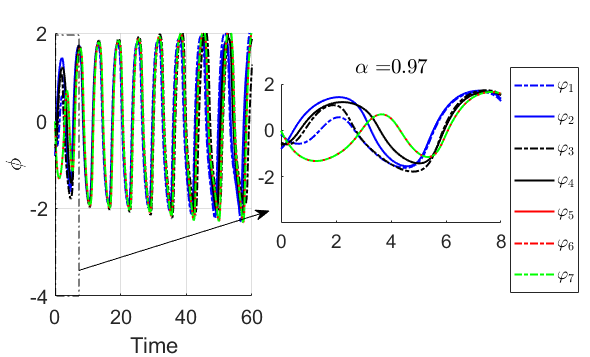}
\caption{Time series of \textsc{FMHNN} model with ring structure (\ref{ndim}) consisting of 7 sub-networks with $p=1$, $d=0.5$, $N=7$  and initial conditions $X_1=(-0.17,-0.80,-0.53,-0.63,-0.016,-0.11,-0.49),~X_2=(-0.39,-0.06,-0.41,-0.29,-0.98,-0.37,-0.34),~\phi=(-0.83,-0.40,-0.66,-0.43,-0.17,-0.12, -0.95)$ for $\alpha = 0.97$ and $\alpha = 0.87$.} \label{fig:N seven}
\end{figure}

\begin{figure}[!h]  
\centering
\includegraphics[width=15cm, height=8cm]{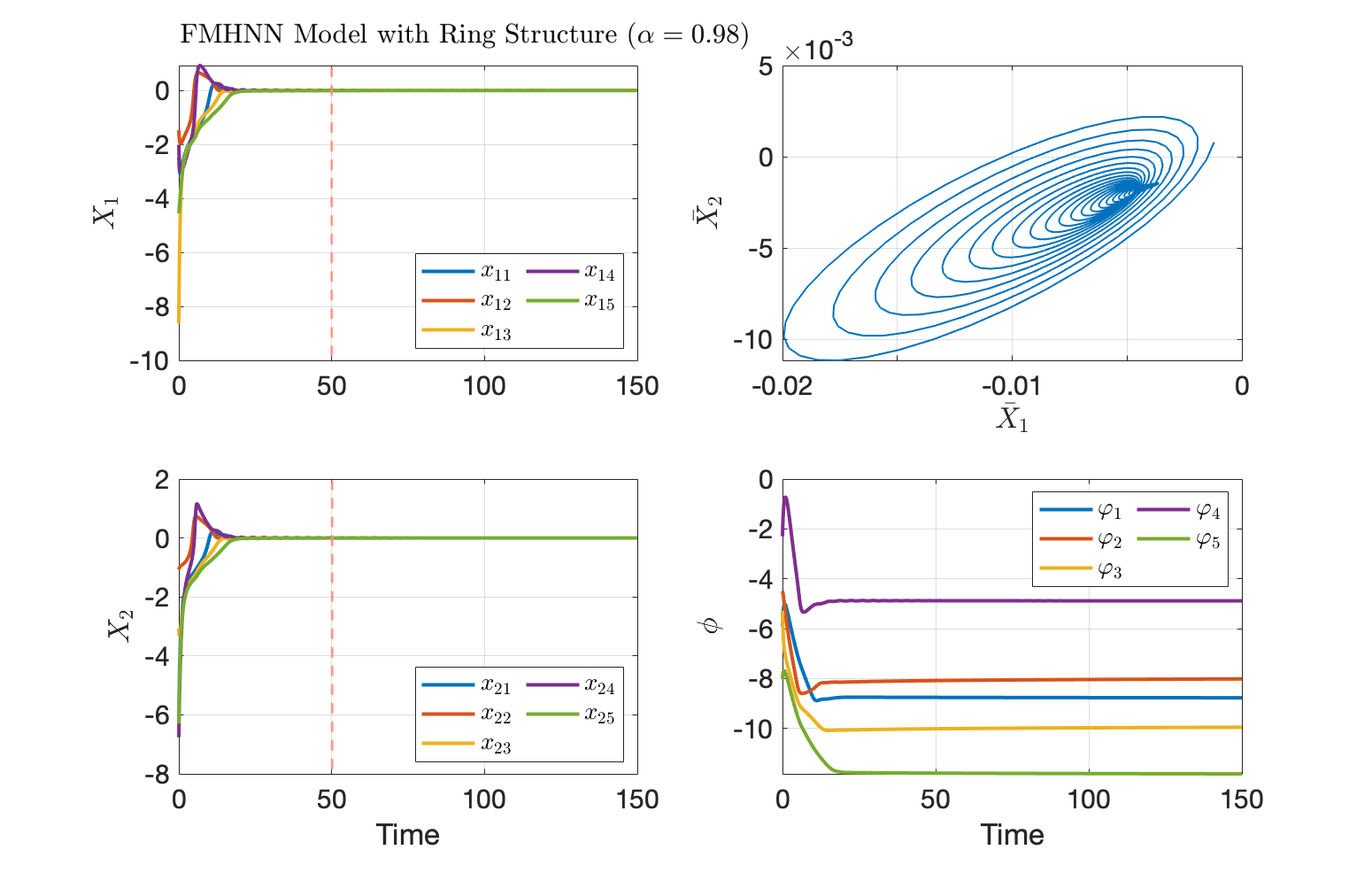}
\includegraphics[width=15cm, height=9cm]{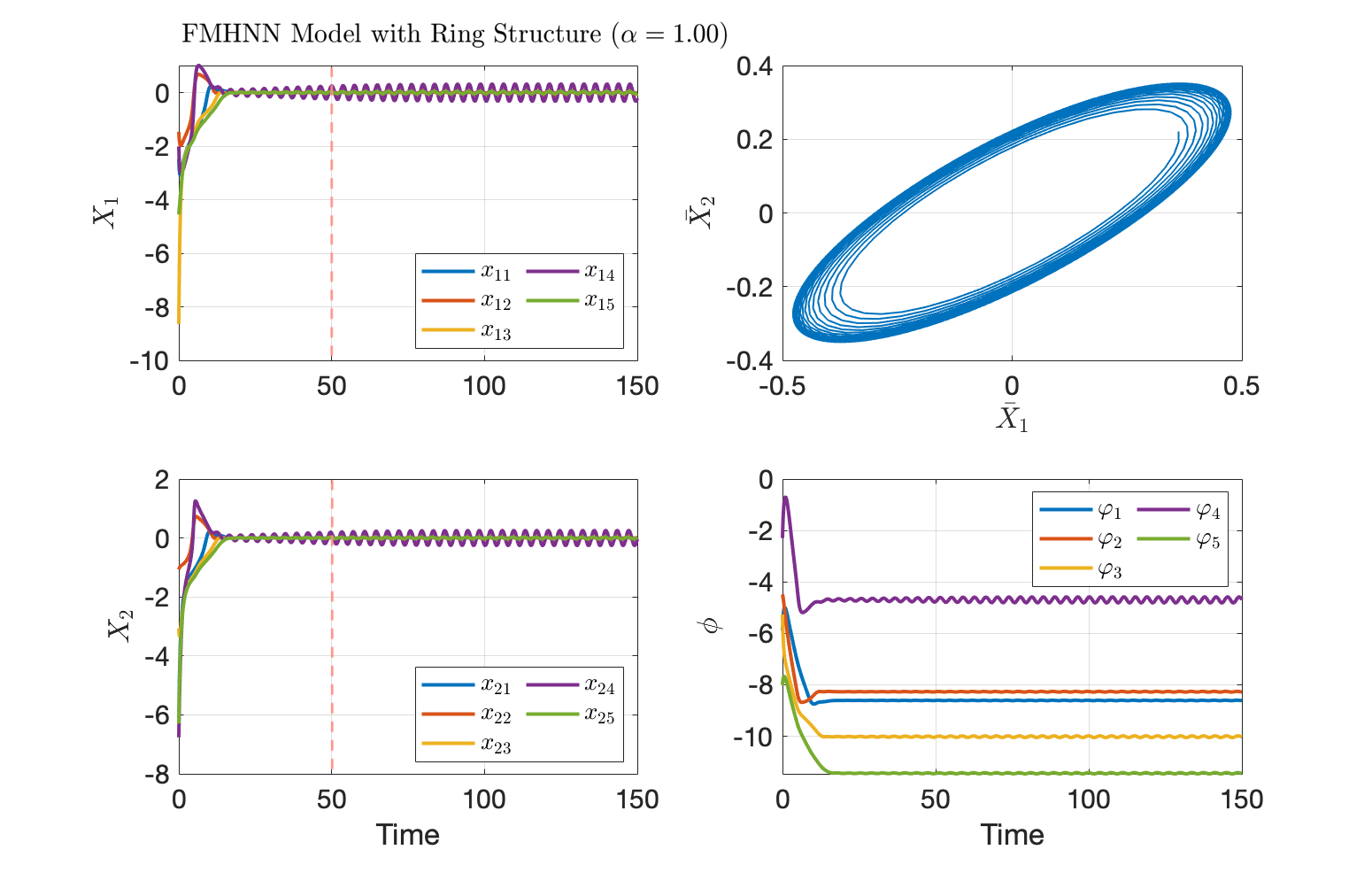}
\caption{Phase diagrams and time series of \textsc{FMHNN} model with ring structure (\ref{ndim}) consisting of 5 sub-networks with $p=1$, $d=0.5$ and initial condition $X_1=(-2.48,-6.12,-5.90,-1.46,-1.07),
~X_2=(-4.48,-8.64,-3.06,-5.27,-2.01),
~\phi=(-6.76,-2.30,-4.55,-6.30,-8.02)$ where $X_1 = [x_{1i}], ~ X_2 = [x_{2i}],~ \phi = [\varphi_{i}]$ for $i=1, \cdots, 5$. $\bar{X}_1=  \frac{1}{5}\sum^5_{i=1} x_{1i},~\text{and}~ \bar{X}_2= \frac{1}{5}\sum^5_{i=1} x_{2i}$ are average potentials between two neurons, calculated for $t \geq 50$. Further details are given in the titles and legends. } \label{fig:integer-nD}
\end{figure}

\begin{figure}[!h] 
\centering
\includegraphics[width=11cm, height=6cm]{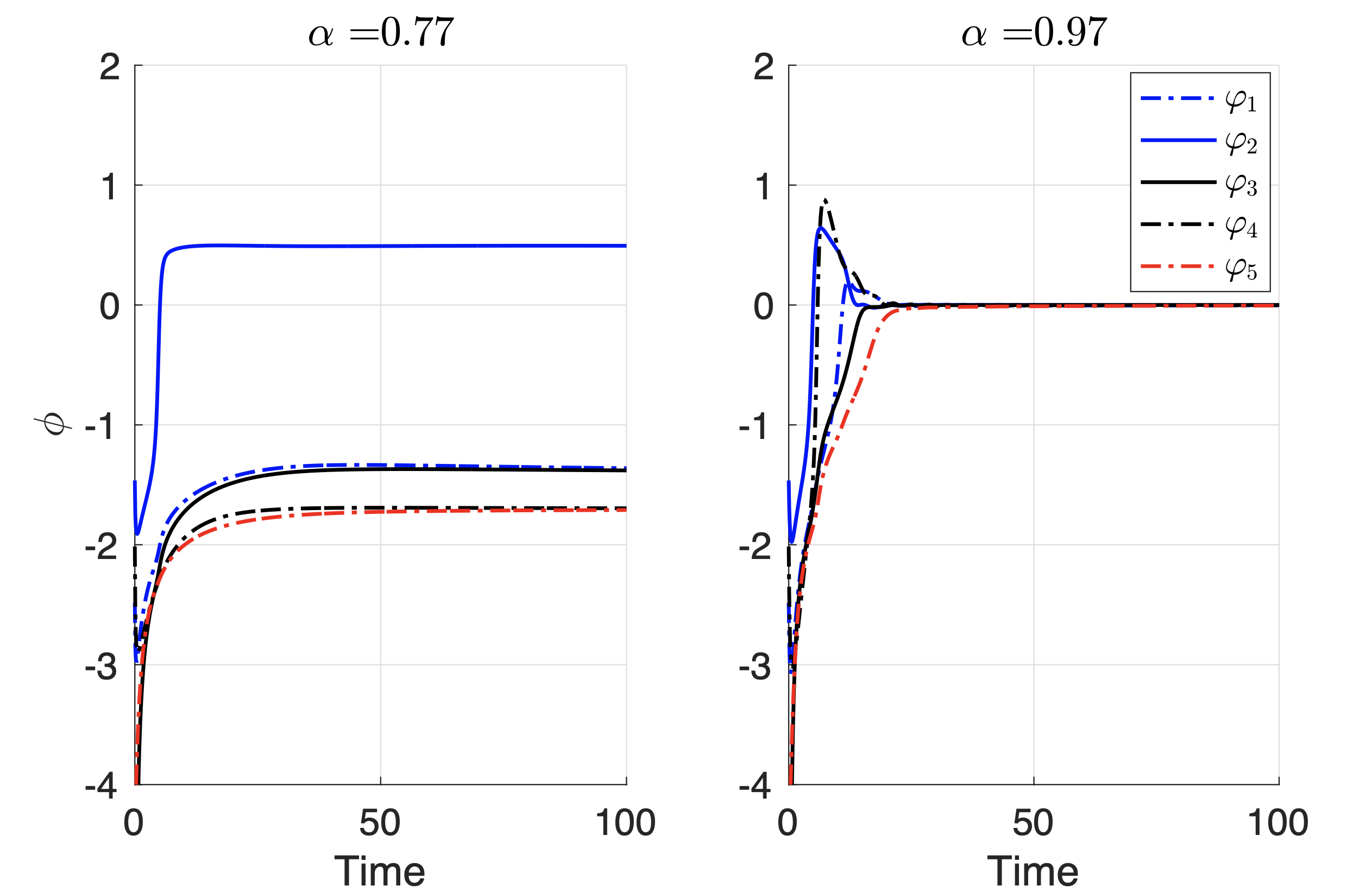}
\caption{Time series of \textsc{FMHNN} model with ring structure (\ref{ndim}) consisting of 5 sub-networks with $p=1$, $d=0.5$, $N=5$ and initial conditions $X_1=(-2.48,-6.12,-5.90,-1.46,-1.07), ~X_2=(-4.48,-8.64,-3.06,-5.27,-2.01),
~\phi=(-6.76,-2.30,-4.55,-6.30,-8.02)$ for $\alpha = 0.97$ and $\alpha = 0.87$} \label{fig:NDIME}
\end{figure}

\textbf{Bifurcation Analysis}. Here we would study the stability criteria further through investigating the influences of the initial conditions and the fractional-order parameter $\alpha$ on our \textsc{FMHNN} model (\ref{eq:fmhnn_2_N}).

Again, we set $b_1=-0.1,~ b_2=2.8,~ b_3=-3 ~\text{and} ~b_4=4$ in system (\ref{eq:fmhnn_2_N}), as assumed by Chen et al \cite{chen2019coexisting}.
Taking $\alpha$ as the bifurcation parameter when  $\alpha \in[0.96, 1]$ and the initial value $[0,10^{-6},0]$, we performed the bifurcation analysis. Figure (\ref{fig7}) shows the bifurcation diagram of the $\alpha - \phi$ plane, indicating that with the increase of $ \alpha $ value the trajectory of \textsc{FMHNN} evolves into unbounded behavior, chaos, eight-period cycles, four-period and finally to two-period cycles. Also according to the Lyapunov exponent shown in figure (\ref{fig7}), the system can not be asymptotically stable for the given interval value of $\alpha$. 
\begin{figure}[!h] 
\centering
\includegraphics[width=7.5cm,height=5.9cm]{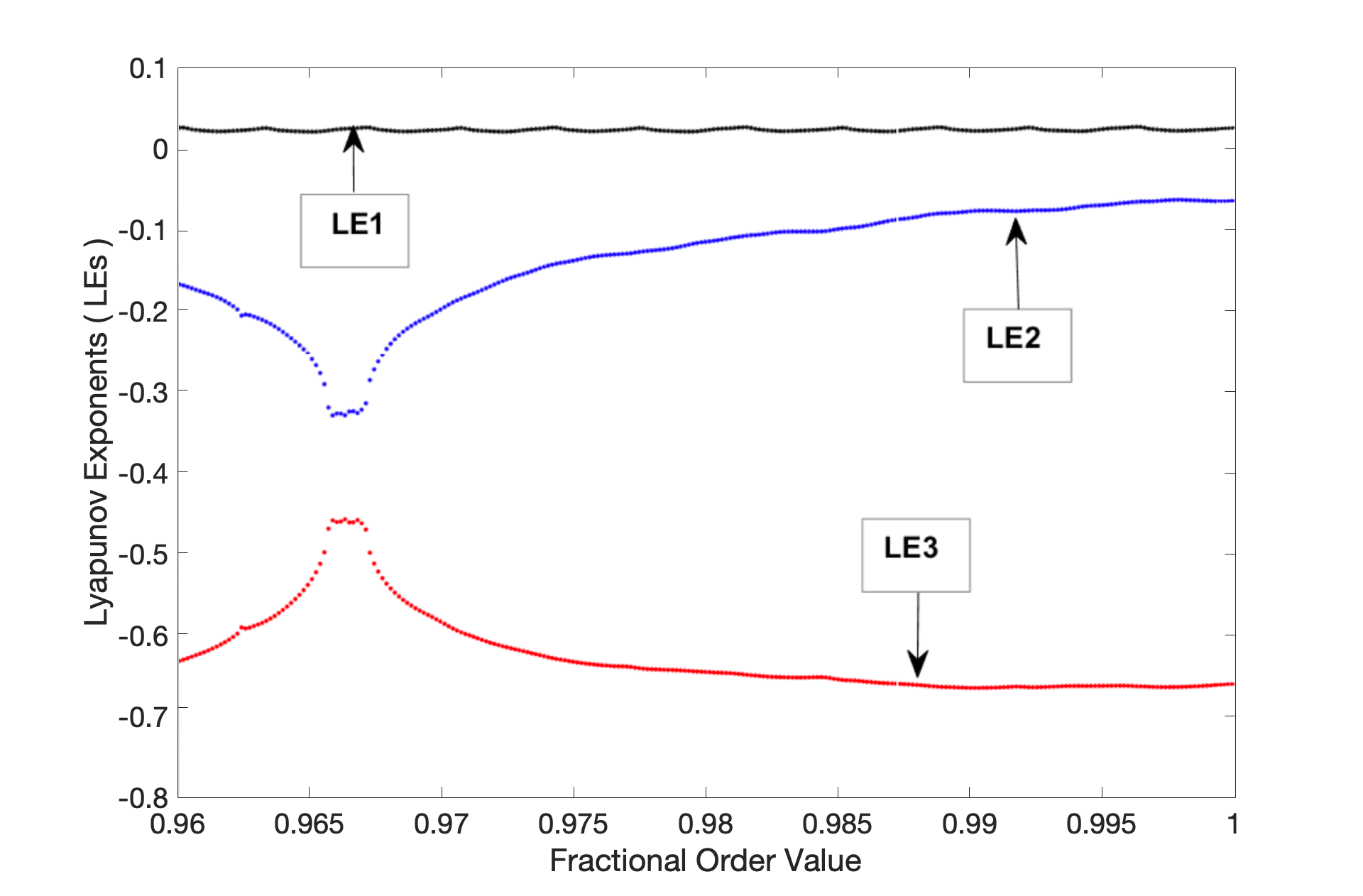}
\includegraphics[width=7cm,height=5.8cm]{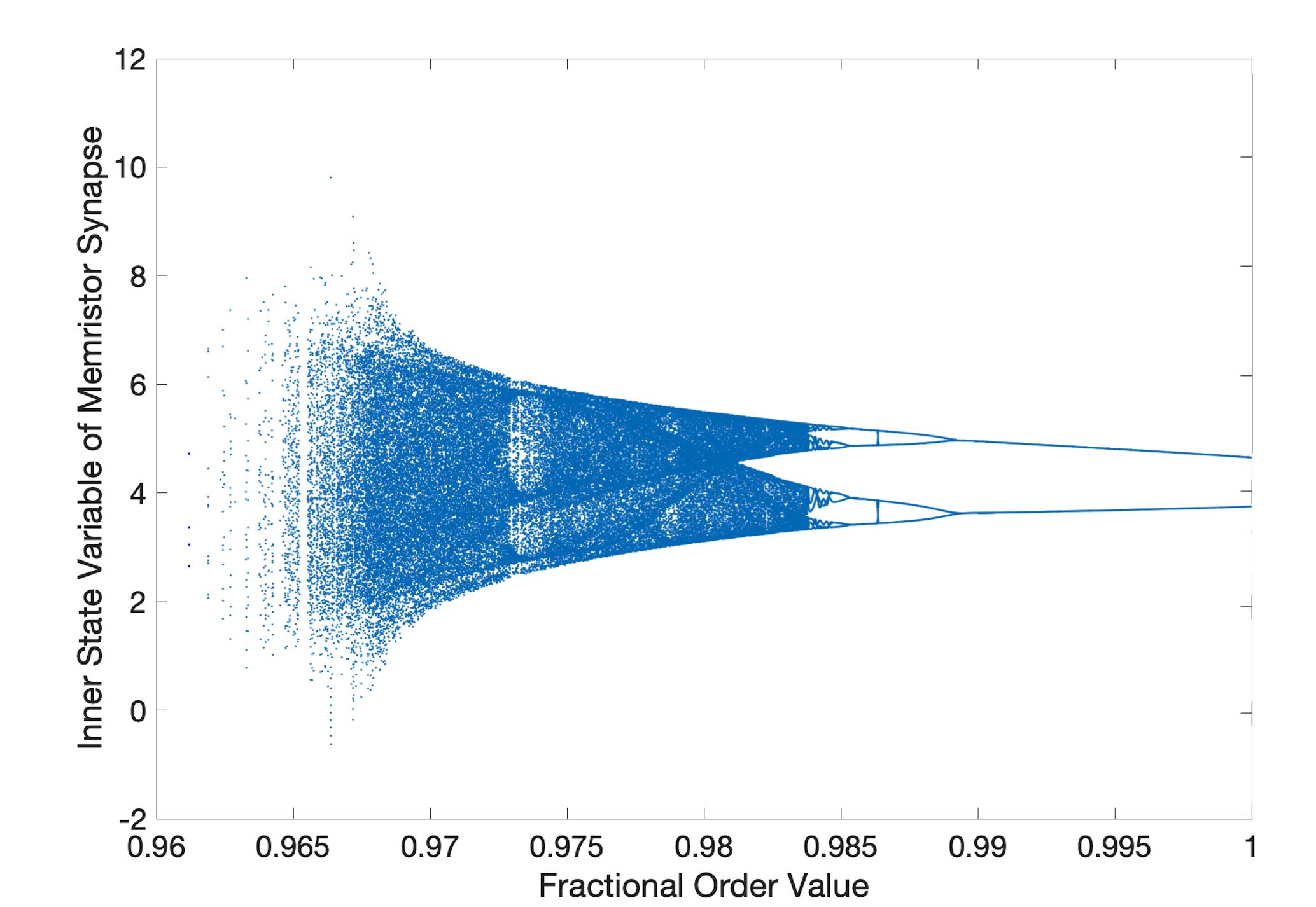}
\caption{The bifurcation diagram and Lyapunov exponents of \textsc{FMHNN} model (\ref{eq:fmhnn_2_N}) on the $\alpha - \phi$ plane over $0.96 \leq \alpha \leq 1$ interval.} 
\label{fig7}
\end{figure}

Next, we examined the impact of initial conditions on the dynamical transition of the model output by performing three independent simulations under different initial conditions: $X_{\text{blue}} = [0.5, -1, 1]$, the blue trajectory, $X_{\text{red}} = [0.5, -0.5, 0.5]$, the red trajectory, and $X_{\text{black}} = [0.5, -0.7, 0.7]$, the black trajectory. From  Fig. (\ref{fig:ALLBIF}), we observed that 
under the $X_{\text{red}}$ initial value, the trajectory of the \textsc{FMHNN} model (\ref{eq:fmhnn_2_N}) evolves from chaotic state into four-period cycles, then two-period and eventually into one-period cycles, while  under the $X_{\text{black}}$ initial value, we have different trajectories. The bifurcation patterns evolves from chaotic state into four-period cycles, two and then one-period, again two-period, and finally one-period cycles. Similar patterns is seen for the $X_{\text{blue}}$ initial value (Fig. \ref{fig:ALLBIF}). This indicates the initial conditions could have an essential influence on the bifurcation behavior of the model.

\begin{figure}[!h] 
\centering
\includegraphics[width=14.5cm, height=6cm]{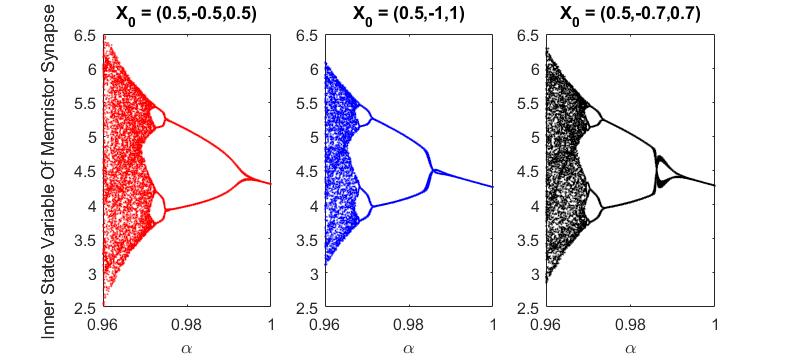}
\caption{The bifurcation diagram of \textsc{FMHNN} model (\ref{eq:fmhnn_2_N}), under three different initial conditions $X_{\text{red}}$ = (0.5, -0.5, 0.5), $X_{\text{blue}}$ = (0.5, -1, 1) and $X_{\text{black}}$ = (0.5, -0.7, 0.7).} \label{fig:ALLBIF}
\end{figure}

\section{Discussion and conclusion}

A memristor is an electrical component restricting or regulating the flow of electrical current in a circuit. This component adds the magnetic flux to the membrane potential of a neuron model, enabling the system to remember the previous amount of charge that has flowed through it. Our FMHNN model extends this memory characteristic to keep the effect of all previous steps in the system. Due to the incorporated memory feature, the model shows more residence to changes compared to the MHNN model. The memory concept, coming from fractional derivative indeed, dampens out fluctuations in the model dynamic. Increasing the memory feature, $\alpha \rightarrow 0$, would lead the system as a whole settles on an equilibrium quicker (Fig. \ref{fig:alpha_comp}). 

Our theoretical calculation indicates that the FMHNN model has a broader range in stability than the MHNN model. (Theory \ref{th:2D_stability}). 
Under the same parameterization as Chen et al \cite{chen2019coexisting}, we verified this. Both models have five stability regions, among which the MHNN model is stable in only one area, while the FMHNN model could be stable in three areas (Table \ref{t1}). Our numerical simulations also agree with this finding (Fig. \ref{fig:case4}), suggesting the memory concept could significantly influence the system dynamics.
\begin{figure}[!h] 
\centering
\includegraphics[width=14.5cm, height=8cm]{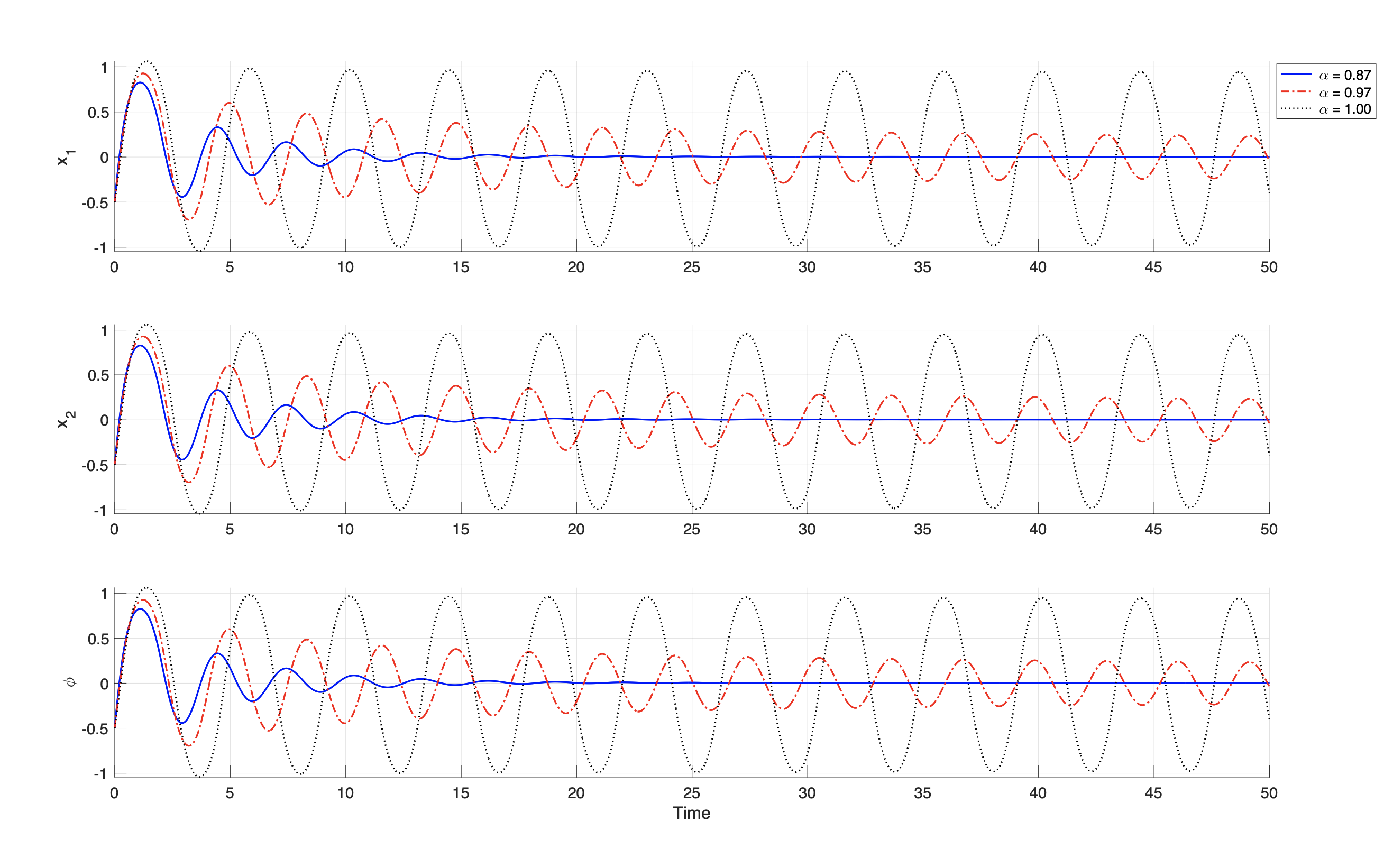}
\caption{Time series of \textsc{FMHNN} model with initial condition $ X_0 $ = (-0.5, 0.25, -5). The parameters are chosen from Chen et al \cite{chen2019coexisting}.} \label{fig:alpha_comp}
\end{figure}

It is noteworthy to notice that the FMHNN could not be stable for all ranges of $\alpha \in (0,1)$. 
To highlight this, we performed bifurcation analyses and estimated Lyapunov exponents. The outputs show that the model experiences several dynamical transitions from unbounded to non-chaotic behaviour by increasing $\alpha$ value from 0.96 to 1 (Fig. \ref{fig7}). When $\alpha =1$, the model is also unstable as the stability condition is violated. Moreover, our simulations shows that the dynamical transitions of model can be affected by changing initial points (Fig. \ref{fig:ALLBIF}).

Next, we expanded the previous model to a general fractional-order network using the ring topology, where each membrane is in ring structure with $n$ sub-networks, improving the synchronization within the synapse and increasing the network dimension. We obtained the stability criteria, highlighting the importance of the synchronization factor in the network. In effect, we found that the stability does not only depend upon the memory element, $\alpha$, but it depends upon the number of sub-networks, $n$, in membranes  (Table \ref{t3j}). For $\alpha=1$ also, the model would have stable limit cycles (Remark \ref{remark}). 

Within the same parameterization as the two-neuron network model, we tested our calculations for the fractional-order network including $n$ units of the two-neuron network model which are connected in ring topology. The units are coupled by the nearest neighbourhood neurons with the coupling strength of $d=0.5$. We found that the model cannot be stable when the network size exceeds $n=6$. To check this numerically, we performed several simulations summarized below. 
\begin{itemize}
    \item $n=5$ \& $0<\alpha < 1$, the system is stable (Fig. \ref{fig:integer-nD}),
    \item $n=5$ \& $\alpha = 1$, the system has a stable limit cycle (Fig. \ref{fig:integer-nD}),
    \item $n=7$, the system is unstable (Fig. \ref{fig:NDIME}).
\end{itemize}

Our simulations also suggest that the fractional-order value could affect the model dynamics significantly. In addition to dampening out the fluctuations in the model dynamics (Fig. \ref{fig:N seven}), the $\alpha$ value could shift the system as a whole settles on different equilibrium points (Fig. \ref{fig:NDIME}).

The number of sub-neurons, the units, plays a pivotal role in the stability condition. Neuronal networks could give us a better insight into the existing collective behaviours of neurons by embedding the synapses in models where nodes are coupled in different types of synchronization. Due to the memory characteristic and enriching synchronized features of the fractional-order rig structure network, we think FMHNN models could be employed to design high-density artificial neural networks to examine complex systems like the brain. Although this work makes a solid mathematical framework to control the FMHNN model via stability analyses, there is more room for improvement. Therefore, it would be interesting to evaluate the impact of lag synchronization in the FMHNN model or to consider other types of synchronization, such as partial synchronization, phase synchronization, lag synchronization, cluster synchronization, and so forth in the model.

\section*{Data availability statement}
No data were used to support this study.

\section*{Statements and Declarations}
The authors did not receive support from any organization for the submitted work. Also they declare that they have no known competing financial interests or personal relationships that could have appeared to influence the work reported in this paper.

\section*{Acknowledgment}
MA expresses his appreciation for many enlightening conversations with \textit{Prof. Andrew  J. Irwin} of Dalhousie University during the revision process.

\bibliographystyle{ieeetr}
\bibliography{refs.bib}

\end{document}